\documentclass[12pt]{article}
\usepackage{amsfonts}
\usepackage{epsfig}
\usepackage{graphics}
\usepackage{inputenc}
\usepackage{lineno,hyperref}
\usepackage{amsmath, amssymb,amsthm, mathtools, empheq}
\usepackage{paralist, mathtools}
\usepackage{caption}
\usepackage{subcaption}
\usepackage{ulem}
\usepackage{placeins}
\usepackage{enumerate, comment}
\usepackage{mathrsfs}
\usepackage{mathtools}
\usepackage{xcolor}

\usepackage{authblk}

\usepackage[T2A]{fontenc}
\renewcommand{\baselinestretch}{1.2}

\def \ve{\varepsilon}

\def \div {{\rm div}}
\def \eps {{\varepsilon}}
\numberwithin{equation}{section}

\newtheorem{thm}{Theorem}
\numberwithin{thm}{section}
\newtheorem{rem}[thm]{Remark}
\newtheorem{prop}[thm]{Proposition}
\newtheorem{lem}[thm]{Lemma}
\newtheorem{cor}[thm]{Corollary}

\title{
	Short scale asymptotics  in nonlocal spectral problems
	}

\author[1]{Volodymyr Rybalko}
\author[2]{Antonina Rybalko}
\affil[1] {B. Verkin Institute for Low Temperature Physics and Engineering  \ \ \ \ \ \ \ \ \ \ \ \ \ \ \ \ of the National Academy of Sciences of Ukraine, Kharkiv, Ukraine;\ \ \ \ \ \ \ \ \ \ \ \ 
e-mail: vrybalko@ilt.kharkov.ua, v\_rybalko@yahoo.com}
\affil[2]{Kharkiv National University of Radio Electronics, Kharkiv, Ukraine; \ \ \ \ \ \ \ \ \
e-mail:  antonina.rybalko@nure.ua}

\begin{document}

\maketitle

\begin{abstract} We study the asymptotic behavior of  eigenvalues of nonlocal 
convolution-type operators 
$$
\mathcal{L}_\ve \rho=-\frac{1}{\ve^{d}}\int_\Omega J\left({\frac{|x-y|^2}{\ve^2}}\right)
	\rho(y)dy+
	a(x)\rho(x),
$$
(with a continuous sufficiently fast decaying at infinity function $J\geq 0$) when the scaling parameter $\ve>0$ tends to zero.  
Assuming that the minimum of the function $a(x)$ is attained at a point $x_0\in \overline{\Omega}\subset \mathbb{R}^d$ and is non-degenerate, 
we establish several terms in the asymptotic expansion of the first eigenvalues $\lambda_\ve^{(k)}$ of $\mathcal{L}_\ve$. 
 If $x_0$ is an interior point of the domain $\Omega$, we show that  $\lambda_\ve^{(k)}=a(x_0)-1+\ve \mu^{(k)}+\overline{o}(\ve)$, where $\mu^{(k)}$ are eigenvalues of a harmonic oscillator (obtained by 
 the blow up analysis), the corresponding eigenfunctions of $\mathcal{L}_\ve$ are concentrated near $x_0$ on the scale $\ve^{1/2}$. If $x_0\in\partial \Omega$ then the asymptotic expansion 
 of the eigenvalues of $\mathcal{L}_\ve$ has the form $\lambda_\ve^{(k)}=a(x_0)-1+\ve^{2/3} \Lambda_0+ \ve (\mu^{(k)}+\alpha p_0)+\overline{o}(\ve)$. The corresponding eigenfunctions are represented 
 as the product of a function depending on the distance to $\partial\Omega$ and concentrated on the scale $\ve^{2/3}$, and  functions localized on the scale $\ve^{1/2}$ and  asymptotically
 independent of the distance to $\partial\Omega$. The latter functions converge (after rescaling) to eigenfunctions of a harmonic oscillator (in tangential variables), $\mu^{(k)}$ being their eigenvalues. The 
 former function is given, to the leading term, by the scaled Airy function, and  $\Lambda_0$, $\alpha$ and $p_0$ depend on $J(|x|^2)$ and  $a(x)$. The results are obtained by reducing the study to 
 $\Gamma$-convergence of special energy functionals.
\end{abstract}


\bigskip


\section{Introduction}

This work deals with asymptotic spectral analysis of nonlocal operators of the form 
\begin{equation}
	\mathcal{L}_\ve \rho =\frac{1}{\ve^{\kappa}}\left(\int_{\mathbb{R}^d}J\left({\frac{|x-y|^2}{\ve^2}}\right)dy\, \rho(x) - \int_\Omega J\left({\frac{|x-y|^2}{\ve^2}}\right)
	\rho(y)dy\right)+
	m(x)\rho(x)
	\label{main_operator_locper_first} 
\end{equation}
when the scaling parameter $\ve>0$ tends to zero. Such nonlocal operators naturally appear in 
mathematical biology and population dynamics when nonlocal interactions have to be taken into 
account  \cite{CroKim1970}, \cite{HutMarMisVic2003}, \cite{OthDunAlt1988}. In the latter context 
$\rho(x)$ is thought of as the density of population at the location $x$ and $\frac{1}{\ve^{\kappa}} J\left(\frac{|x-y|^2}{\ve^2}\right)$
represents the rate of jumps from location $y$ to location $x$. The
(negative of) two first terms in  \eqref{main_operator_locper_first} 
account for the instantaneous change of density due to individuals leaving the location $x$ and those arriving 
from all other places, respectively. 

 It is known (see e.g. \cite{CorEtgRosWol2008}, \cite{BerCovHoa2016})
%
%
%
%
%
that asymptotic analysis of  such convolution-type nonlocal operators 
leads to a local (differential) operator when the scaling parameter $\ve$ tends to zero. In particular, in the case of a sufficiently fast decaying kernel $J(|z|^2)$ and the critical 
value  of the exponent $\kappa=d+2$, resolvents of the operators \eqref{main_operator_locper_first}  
compactly converge to that of $-\frac{\vartheta}{d}\Delta \rho+m(x) \rho$ as $\ve\to 0$, where $\vartheta$ is the second moment of the kernel $J(|z|^2)$.  
This also yields convergence  of eigenvalues and eigenfunctions. 
A similar result holds when the dispersal kernel contains also a fast oscillating factor, as it appears 
in models  of strongly inhomogeneous media (see e.g. \cite{PiaZhi2017},  \cite{PiaZhi2019}). 
%
In \cite{PiaRyb2024} the authors obtained a homogenized description for the bottom point of the spectra of  convolution type 
operators in the case when the exponent $\kappa$ in the pre-factor $1/\ve^\kappa$ equals the dimension $d$ of the 
space. It was shown that this bottom point  converges, as $\ve\to 0$, to the (unique) additive eigenvalue of an effective Hamilton-Jacobi equation.
This result has  its local counterpart \cite{PiaRyb2016}  concerning the asymptotic behavior of the principal eigenvalue of singularly perturbed 
second-order differential operators of the form $\ve^2 a_{ij}\partial_{x_i x_j} u+\ve b_j \partial_{x_j} u +m(x) u$.  Both works \cite{PiaRyb2024} and \cite{PiaRyb2016}  deal with generic 
non-selfadjoint cases, while in the selfadjoint case the said additive eigenvalue equals  $\min_{x\in\overline\Omega} m(x)$ and much more precise asymptotic formulas can be obtained 
via blow-up analysis. In particular,  \cite{AllPia2002}  studies spectral problems for selfadjoint second-order partial differential operators with eigenfunctions localized on the scale $\ve^{1/2}$ around a point in the domain (the minimum point of the potential if the coefficients have no fine-scale oscillations).

In this work, we stick to the case $\kappa=d$ so that the kernels in \eqref{main_operator_locper_first} are  normalized to the same (unit) total mass.
Like in the case of differential operators, eigenfunctions of \eqref{main_operator_locper_first} concentrate near 
the minimum point of the function $m(x)$.  We show that if the (nondegenerate) minimum of $m(x)$ is an interior point of the domain, then the eigenfunctions 
of  \eqref{main_operator_locper_first} are concentrated on the scale $\ve^{1/2}$ near this point, and blow-up analysis of the problem leads 
to the spectral problem for a harmonic oscillator. Another typical case, when $m(x)$ attains its minimum at the point of the boundary requires more sophisticated 
asymptotic analysis of the problem. It  reveals (again under an appropriate nondegeneracy condition)  oscillations of eigenfunctions on the fast scale  $\ve^{2/3}$ in the direction normal to the boundary and on the slower one $\ve^{1/2}$ in tangential directions.  Moreover, we show that eigenfunctions asymptotically split into the product of a function depending on the distance to the boundary and  another one varying in the tangential directions to the boundary. The former function (given, to the leading term, by the scaled Airy function) is constructed via the boundary layer analysis, while for the latter one we derive the 
spectral problem for a harmonic oscillator in tangential variables.   

The main results of this work are obtained via establishing $\Gamma$-limits of appropriate energy functionals. 
Questions of $\Gamma$-convergence of convolution type energy functionals were addressed  in  \cite{BraPia2021} in the 
case of the critical value $\kappa =d+2$ of the pre-factor $1/\ve^\kappa$. A similar variational problem is obtained also  (via the aforementioned blow up argument)
in the asymptotic analysis of the spectral problem \eqref{main_operator_locper_first}   when the minimum of $m(x)$
is attained at an interior point.  The case when $m(x)$ attains its minimum at the boundary is technically much more involved. We use a factorization method. 
First, exploiting two scale nature of eigenfunctions, we construct a boundary layer solution of an one-dimensional problem obtained by averaging the kernel over the 
(slower) tangential variables. This study is based on a solvability result for a model nonlocal problem on the half line that is of interest itself.
Thus we establish the factor responsible for the (fastest) variations in the normal direction and then single out the energy due to this factor.  
Next we turn to the $\Gamma$-convergence of  (properly rescaled) remaining  energy functionals. These functionals penalize asymptotic 
dimension reduction, and also have a specific difficulty due to the weight factor that asymptotically vanishes on the boundary and at infinity. 
It makes establishing the compactness result (needed for the asymptotic spectral analysis)  quite a challenging task. 
  
The paper is organized as follows. In Section \ref{PSMR}, we state the main results of the work
Theorem \ref{Pershath} and Theorem  \ref{Th22}. Theorem \ref{Pershath}  deals with the case 
when $m(x)$ attains its minimum at an interior point.  The proof of this theorem is presented in Section \ref{secInter}. The primary finding of this study  
is Theorem  \ref{Th22}  dealing with the case when the minimum of $m(x)$ is located on the boundary.  The proof of  Theorem \ref{Pershath} constitutes the remaining Section \ref{SamaSamaSecttsiya}.  
It consists of Subsections \ref{Consruct} and  \ref{BLS} dealing with the construction of the boundary layer solution, Subsection  \ref{Factor_Repr}, where we use the factorization method 
to reduce study to the convergence of special energies, and Subsections \ref{upperbound}, \ref{lowerboundineq} and \ref{CompactThM} establishing the $\Gamma$-convergence and  compactness results for those energies, respectively.

\bigskip

\noindent
{\bf ACKNOWLEDGEMENTS.}
The research of the first author is partially supported by the SSF grant
for Ukrainian scientists Dnr UKR 22-0004,  UKR 24-0023, 
and by the Troms{\o}  Research Foundation 
project "Pure Mathematics in Norway". The authors are grateful to  professor A. Piatnitski for introduction into the subject and fruitful discussions.


\section{Problem setup and main results}
\label{PSMR}
We consider the spectral problem
\begin{equation}
	\mathcal{L}_\ve \rho_\ve = \lambda_\ve \rho_\ve(x)\quad \text{in}\ \Omega,
	\label{start_eq} 
\end{equation}
where $\Omega$ is a bounded domain  in $\mathbb{R}^d$ with $C^3$-boundary, and $\mathcal{L}_\ve$ is the nonlocal operator 
given by
\begin{equation}
	\mathcal{L}_\ve \rho =-\frac{1}{\ve^{d}}\int_\Omega J\left({\frac{|x-y|^2}{\ve^2}}\right)
	\rho(y)dy+
	a(x)\rho(x)
	\label{main_operator_locper} 
\end{equation}
with $\ve>0$ being a small parameter characterizing the typical range of interactions in the system.
We assume that the following conditions are fulfilled
\begin{equation}
	J\in C(\mathbb{R}_+), 
	\quad  a\in C^2(\overline{\Omega}),
		\label{umova_raz}
\end{equation}
\begin{equation}
	\label{umova_dva}
	\quad J(0)>0, \quad \ J(z)\geq 0 \ \  \forall z\in \mathbb{R}_+, 
\end{equation}
\begin{equation}
		\label{umova_try}
	\int\limits_{\mathbb{R}^d} J(|\xi|^2)d\xi=1, \quad \vartheta:=\int\limits_{\mathbb{R}^d} J(|\xi|^2) |\xi|^2 d\xi< \infty. 
	\end{equation} 

		Suppose that $a(x)$ attains its minimum over $\overline{\Omega}$ at exactly one point $x_0$, and for convenience, we
will assume that $x_0=0$.  Then, as shown in \cite{BerCovHoa2016} (see also \cite{PiaRyb2024} for nonself-adjoint convolution-type operators), the principal eigenvalue converges to $m(0)$, where 
		\begin{equation}
	m(x)=a(x)-1.
			\label{min_m} 
		\end{equation}
		Under the additional non-degeneracy assumption on the minimum, we establish much more precise asymptotic formulae for the eigenvalues at the bottom part of the spectrum. The results depend  essentially on whether the minimum of $m(x)$ is an interior point of $\Omega$ or it lies on the boundary. 
				If $0\in\Omega$
		then we suppose the strict positivity of the Hessian,
		\begin{equation}
			\partial^2_{x_ix_j} m(0) \xi_i\xi_j>0 \quad \forall \xi\in\mathbb{R}^{d}\setminus\{0\},
			\label{nond}
		\end{equation}
		and prove the following result.
		\begin{thm}
		\label{Pershath}
			If, under conditions 
			\eqref{umova_raz}-\eqref{umova_try}, the function $m(x)$ attains its minimum at the unique point  $0\in\Omega$ and \eqref{nond} is satisfied, then the first $n$  eigenvalues of $\mathcal{L}_\ve$ (for any $n$ fixed) are given by the asymptotic formula
	\begin{equation*}
		\lambda_\ve^{(k)}=m(0)+\ve\mu^{(k)}+\bar o(\ve),
	\end{equation*}
	where $\mu^{(k)}$ are the eigenvalues of harmonic oscillator, 	
	\begin{equation}
			\mathcal{L}_0 \rho=-\frac{\vartheta}{2d} \Delta \rho (\xi)+
			\frac{1}{2}\partial^2_{x_ix_j} m(0)\xi_i\xi_j \, \rho (\xi) \quad {\rm in} \ \mathbb{R}^d.
				\label{harm_osc}
		\end{equation} 
	Moreover, the corresponding rescaled and properly normalized eigenfunctions
	$ \rho_\ve (\xi/\sqrt{\ve})$    of $\mathcal{L}_\ve$ $L^2$-strongly converge, up to a subsequence, to the eigenfunctions of $\mathcal{L}_0$.
			\label{Th21}
		\end{thm}
		
In the case when the minimum point  $x_0=0\in \partial \Omega$, we pass to the coordinate system such that the $x_d$-axis is oriented along the inner normal, 
while the other axes are oriented along the principal directions of the surface  $\partial \Omega$ at zero. Then, the non-degeneracy of the minimum means that 
		\begin{equation}
			\alpha:=\partial_{x_d} m(0)>0, 
			\label{2.9}
		\end{equation}
		\begin{equation}	
			\alpha \kappa_i\xi_i^2+ \partial^2_{x_i x_j} m(0)\xi_i \xi_j>0 \quad \forall \xi\in\mathbb{R}^{d-1}\setminus\{0\},
\label{2.10}			
		\end{equation}
where $\kappa_i$ denote the principal curvatures of $\partial\Omega$.
		
	The main result of this work is 
	\begin{thm}
		\label{Th22}
		Assume that 	\eqref{umova_raz}--\eqref{umova_try} hold, moreover, $J$ is a $C^2$-function with compact support. Assume also that the function $m(x)$ attains its minimum at the unique point  $0\in\partial\Omega$ and \eqref{2.9},  \eqref{2.10} are satisfied, then the first $n$  eigenvalues of $\mathcal{L}_\ve$ (for any $n$ fixed) are given by the asymptotic formula
\begin{equation*}
		\lambda_\ve^{(k)}=m(0)+\ve^{2/3}\Lambda_0+\ve(\mu^{(k)} +\alpha q_0) +\bar o(\ve),
\end{equation*}
where $\Lambda_0$ is the principal eigenvalue 
		of the equation 
		\begin{equation}
			-\frac{\vartheta}{2d} \, \phi_0''(q) + \alpha q \phi_0(q)=\Lambda_0 \phi_0(q), \quad q>0, 
			\label{Airy_sln}
		\end{equation}
with the boundary condition $\phi_0(0)=0$;
$\mu^{(k)}$ are the eigenvalues of the $(d-1)$-dimensional harmonic oscillator 
\begin{equation}
			\mathcal{L}_0 \psi=-\frac{\vartheta}{2d} \Delta \psi (\xi)+\frac{1}{2}\left(\alpha \kappa_i\xi_i^2 +
			\partial^2_{x_ix_j} m(0)\xi_i\xi_j\right)\psi (\xi) \quad {\rm in} \ \mathbb{R}^{d-1};
			\label{harm_osc_bound}
\end{equation} 
	 $q_0$ is defined in 	Corollary \ref{cor72}.
		\label{Th22}
	\end{thm}
			
Notice that the operator $\mathcal{L}_\ve$, being a compact perturbation of the operator of multiplication by the function $a(x)$, has the essential spectrum $a(\overline{\Omega})$ 
and has a (possibly empty) discrete bounded set of eigenvalues with finite multiplicities in $\mathbb{R}\setminus a(\overline{\Omega})$. 
If there are, say $n$, eigenvalues less than $\min\limits_{\overline{\Omega}} a(x)$ then they can be identified via the Courant minimax principle
\begin{equation*}
	\lambda_\ve^{(k)}=\max\limits_{{\rm codim}E=k-1}
	\min\limits_{\rho\in E\setminus \{0\}}
\frac{1}{\|\rho\|^2_{L^2(\Omega)}}\left(\int\limits_{\mathbb{R}^d} \int\limits_{\mathbb{R}^d}
J\left(\frac{|x-y|^2}{\ve^2}\right)(\rho(x)-\rho(y))^2 \frac{dxdy}{2\ve^d}
+\int\limits_{\mathbb{R}^d} m(x)\rho^2(x)dx\right),
\end{equation*} 	
where the maximum is taken over subspaces $E$ of $L^2(\Omega)$, $\|\,\cdot\,\|_{L^2(\Omega)}$ denotes the standard norm in $L^2(\Omega)$ and functions $\rho\in L^2(\Omega)$ are assummed to be continued in the whole $\mathbb{R}^d$ by setting $\rho(x)=0$ in $\mathbb{R}^d\setminus \Omega$. This variational approach to the spectral problem allows us to employ the asymptotic technique of $\Gamma$-convergence to prove 
Theorem \ref{Th21} and Theorem \ref{Th22}.

\section{Eigenfunction concentration at an interior point}
\label{secInter}

In this section we prove Theorem \ref{Th21}, i.e. we consider the case when $m(x)$($=a(x)-1$) attains its strict global minimum at $x_0=0\in \Omega$, 
and non-degeneracy condition 
\eqref{nond} is satisfied. We guess that eigenfunctions are localized at the scale $\sqrt{\ve}$ and study
 the asymptotic behavior, as $\ve \to 0$, of the following 
rescaled via changing of variables $\xi=x/\sqrt{\ve}$  functionals
 \begin{equation}
F_\ve(\rho_\ve)=\frac{1}{2\ve^{d/2+1}}\int\limits_{\mathbb{R}^d}
\int\limits_{\mathbb{R}^d}
J\left(\frac{|\xi-\tilde \xi|^2}{\ve}\right)
(\rho_\ve(\xi)-\rho_\ve(\tilde \xi))^2d\xi d\tilde \xi
+\frac{1}{\ve}\int\limits_{\mathbb{R}^d} \left(m(\sqrt{\ve}\xi)-m(0)+\ve\right)\rho^2_\ve(\xi)d\xi,
\label{funcFve} 	
\end{equation}
where $\rho_\ve\in L^2(\mathbb{R}^d)$ and vanish on $\mathbb{R}^d\setminus \frac{1}{\sqrt{\ve}}\Omega$.  


We prove the following two theorems Theorem \ref{Th31} and Theorem \ref{Th32} about the $\Gamma$-convergence and the compactness. 
\begin{thm} The functionals $F_\ve(\rho_\ve)$ given by \eqref{funcFve}  $\Gamma$-converge to the functional 
	  \begin{equation}
		F_0(\rho)=\frac{\vartheta}{2d}\int\limits_{\mathbb{R}^d}
		|\nabla \rho(\xi)|^2d\xi+\int\limits_{\mathbb{R}^d}
		\left(1+\frac{1}{2}\partial^2_{x_ix_j} m(0)\xi_i\xi_j\right)\rho^2(\xi)d\xi
		\label{funcF00} 	    
	\end{equation}
	 with respect to the  convergence in $L^2(\mathbb{R}^d)$.
		\label{Th31}
\end{thm}

\begin{thm} Assume that $F_\ve(\rho_\ve)\leq C$. Then functions $\rho_\ve(\xi)$ converge strongly in $L^2(\mathbb{R}^d)$, up to a subsequence, to a function $\rho_0\in H^1(\mathbb{R}^d)$ and such that $F_0(\rho_0)<\infty$.
	\label{Th32}
\end{thm}
\noindent Combining these results, and using the Courant minimax principle we establish the statement of Theorem \ref{Th21}.

The proof of Theorem \ref{Th32} follows the lines of Lemma 3.5 in \cite{PiaRyb2024}. While that Lemma deals with the case of a 
fixed bounded domain, the domains $\frac{1}{\sqrt{\ve}}\Omega$ expand to the whole space  $\mathbb{R}^d$. However, this is compensated by the second term in \eqref{funcFve} that provides the following bound
\begin{equation*}
c\int\limits_{\mathbb{R}^d}\left(1+|\xi|^2\right) \rho^2_\ve(\xi)d\xi\leq
\frac{1}{\ve}\int\limits_{\mathbb{R}^d} \left(m(\sqrt{\ve}\xi)-m(0)+\ve\right)\rho^2_\ve(\xi)d\xi
\leq C,
\end{equation*}
where $c>0$ and $C$ are independent of $\ve$.

Next we proceed with the
\begin{proof}[Proof of Theorem \ref{Th31}]
	First we show the $\Gamma$-$\limsup$ inequality, i.e. for any $\rho_0\in H^1(\mathbb{R}^d)$ there exists a sequence of functions  $\rho_\ve$ vanishing on $\mathbb{R}^d\setminus \frac{1}{\sqrt{\ve}}\Omega$, converging to $\rho_0$ in  $L^2(\mathbb{R}^d)$, and such that
		\begin{equation}
		\lim_{\ve\to 0} F_\ve(\rho_\ve)= F_0(\rho_0). 
		\label{3.7}\quad 
	\end{equation}
It sufficies to prove \eqref{3.7} in the case when $\rho_0\in C_c^\infty(\mathbb{R}^d)$ thanks to 

	\begin{lem} 
	For any 
	function $\rho_\ve\in H^1_0\big(\frac{1}{\sqrt{\ve}}\Omega\big)$
the following bound holds  
	\begin{equation*}
		\begin{aligned}	
		\frac{1}{2\ve^{d/2+1}} \int\limits_{\mathbb{R}^d}
			\int\limits_{\mathbb{R}^d}
			J\left(\frac{|\xi-\tilde \xi|^2}{\ve}\right)
			(\rho_\ve(\xi)-\rho_\ve(\tilde \xi))^2d\xi d\tilde \xi \, \leq \, \frac{\vartheta}{2}\int\limits_{\mathbb{R}^d} |\nabla\rho_\ve(\xi)|^2 \, d \xi.
	\end{aligned}
	\end{equation*}
	\label{EP}  
\end{lem}
\begin{proof}
First we perform the change of variables $\xi-\tilde \xi=\sqrt{\ve}\tilde\zeta$ and get
	\begin{equation}
	\begin{aligned}	
		\frac{1}{2\ve^{d/2+1}}&\int\limits_{\mathbb{R}^d}
		\int\limits_{\mathbb{R}^d}
		J\left(\frac{|\xi-\tilde \xi|^2}{\ve}\right)
		(\rho_\ve(\xi)-\rho_\ve(\tilde \xi))^2 d\xi d\tilde \xi \\
		&=\frac{1}{2\ve}\int\limits_{\mathbb{R}^d} d\tilde\xi 
		\int\limits_{\mathbb{R}^d}
		J\left(|\tilde \zeta|^2\right)
		(\rho_\ve(\tilde\xi+\sqrt{\ve}\tilde\zeta)-\rho_\ve(\tilde \xi))^2	d\tilde \zeta. 
		\label{3.6} 
	\end{aligned}
\end{equation}
Representing $\rho_\ve(\tilde\xi+\sqrt{\ve}\tilde\zeta)-\rho_\ve(\tilde \xi)=
\sqrt{\ve}\int\limits_0^1|\nabla\rho_\ve(\tilde\xi+\sqrt{\ve}t\tilde\zeta)|^2 \, dt $
 and applying the Cauchy-Schwarz inequality we obtain   
$|\rho_\ve(\tilde\xi+\sqrt{\ve}\tilde\zeta)-\rho_\ve(\tilde \xi)|^2\leq
\ve |\tilde \zeta|^2 \int\limits_0^1|\nabla\rho_\ve(\tilde\xi+\sqrt{\ve}t\tilde\zeta)|^2 \, dt $. 
Then we estimate the right-hand side of \eqref{3.6} as follows 
\begin{equation*}
	\begin{aligned}
		\frac{1}{2\ve}&\int\limits_{\mathbb{R}^d} d\tilde\xi 
		\int\limits_{\mathbb{R}^d}
		J\left(|\tilde \zeta|^2\right)
		(\rho_\ve(\tilde\xi+\sqrt{\ve}\tilde\zeta)-\rho_\ve(\tilde \xi))^2	d\tilde \zeta \\
	&	\leq	
		\frac{1}{2}\int\limits_0^1 dt 
		\int\limits_{\mathbb{R}^d}
J\left(|\tilde\zeta|^2\right) |\tilde \zeta|^2 d\tilde \zeta 
\int\limits_{\mathbb{R}^d} |\nabla\rho_\ve(\tilde\xi+\sqrt{\ve}t\tilde\zeta)|^2 \, d\tilde \xi=	\frac{\vartheta}{2}
\int\limits_{\mathbb{R}^d} |\nabla\rho_\ve(\xi)|^2 \, d \xi.
\end{aligned}
\end{equation*}
	\end{proof} 

\noindent
{\it Proof of Theorem \ref{Th31} (continued).} Given $\rho_0\in C_c^\infty(\mathbb{R}^d)$, 
we set $\rho_\ve(\xi)=\rho_0(\xi)$ when $\xi\in \frac{1}{\sqrt{\ve}}\Omega $ and $\rho_\ve(\xi)=0$ otherwise. Notice that for sufficiently small $\ve$ we have $\rho_\ve(\xi)=\rho_0(\xi)$ on $\mathbb{R}^d$, then
introducing $\zeta=\frac{1}{\sqrt{\ve}}(\tilde\xi-\xi)$ and expanding  
\begin{equation*}
	\rho_0(\tilde\xi)-\rho_0(\xi)=\rho_0(\xi+\sqrt{\ve}\zeta)-\rho_0(\xi)=\sqrt{\ve}\nabla\rho_0(\xi)\cdot \zeta +
\ve\int\limits_0^1	
	\partial^2_{\xi_j \xi_k}  \rho_0(\xi+\sqrt{\ve}t\zeta)\zeta_j\zeta_k(1-t)dt
\end{equation*}	
 we derive the asymptotic formula for the first term of the functional \eqref{funcFve} 
	\begin{equation*}
	\begin{aligned}	
		\frac{1}{2\ve^{d/2+1}}&\int\limits_{\mathbb{R}^d}
		\int\limits_{\mathbb{R}^d}
		J\left(\frac{|\xi-\tilde \xi|^2}{\ve}\right)
		(\rho_\ve(\xi)-\rho_\ve(\tilde \xi))^2 d\xi d\tilde \xi \\ 
		& = 	\frac{1}{2d}\int\limits_{\mathbb{R}^d}
		J\left(|\zeta|^2)\right)|\zeta|^2 d\zeta 
		 	\int\limits_{\mathbb{R}^d}|\nabla\rho_0(\xi)|^2
		 	d\xi +O(\sqrt{\ve}).
	\end{aligned}
\end{equation*}
The limit transition in second term of  \eqref{funcFve} is straightforward. 

It remains to demonstrate the $\Gamma$-$\liminf$ inequality, i.e. 
\begin{equation}
	\liminf_{\ve\to 0} F_\ve(\rho_\ve)\geq F_0(\rho_0), \quad \text{for any } \rho_\ve\to \rho_0\ \text{in}
	\ L^2(\mathbb{R}^d), \text{as} \  \ve\to 0.
	\label{3.12}
\end{equation}
Thanks to Lemma \ref{EP} it is sufficient to establish \eqref{3.12} for $\rho_0\in C_c^\infty(\mathbb{R}^d)$. 
To this end we write
	\begin{equation}
	\begin{aligned}	
	\frac{1}{2\ve^{d/2+1}}&\int\limits_{\mathbb{R}^d}
\int\limits_{\mathbb{R}^d}
J\left(\frac{|\xi-\tilde \xi|^2}{\ve}\right)
(\rho_\ve(\xi)-\rho_\ve(\tilde \xi))^2 d\xi d\tilde \xi \\ 
& =	\frac{1}{2\ve^{d/2+1}}\int\limits_{\mathbb{R}^d}
\int\limits_{\mathbb{R}^d}
J\left(\frac{|\xi-\tilde \xi|^2}{\ve}\right)
(\rho_\ve(\xi)-\rho_0(\xi)+\rho_0(\tilde\xi)-\rho_\ve(\tilde \xi))^2 d\xi d\tilde \xi \\ 
& +		\frac{1}{2\ve^{d/2+1}}\int\limits_{\mathbb{R}^d}
\int\limits_{\mathbb{R}^d}
J\left(\frac{|\xi-\tilde \xi|^2}{\ve}\right)
\big(\rho_0(\xi)-\rho_0(\tilde\xi)\big)^2 d\xi d\tilde \xi\\
&+	\frac{2}{\ve^{d/2+1}}\int\limits_{\mathbb{R}^d}
\int\limits_{\mathbb{R}^d}
J\left(\frac{|\xi-\tilde \xi|^2}{\ve}\right)
\big(\rho_0(\xi)-\rho_0(\tilde\xi)\big)d\tilde \xi \,  \big(\rho_\ve(\xi)-\rho_0(\xi)\big)d\xi.
\label{3.14}
	\end{aligned}
\end{equation}
The first term in the right hand side of \eqref{3.14} is non-negative; the second term has already been dealt with in the proof of $\Gamma$-$\limsup$ inequality; finally, the last term tends to zero since $\rho_\ve-\rho_0\to 0$ in $L^2(\mathbb{R}^d)$ while 	$\frac{2}{\ve^{d/2+1}}
\int\limits_{\mathbb{R}^d}
J\left(\frac{|\xi-\tilde \xi|^2}{\ve}\right)
(\rho_0(\xi)-\rho_0(\tilde\xi))d\tilde \xi$ remains bounded in $L^2(\mathbb{R}^d)$  as $\ve\to 0$. Indeed, 
arguing as in the proof of $\Gamma$-$\limsup$ inequality, we can write 
	\begin{equation*}
		\frac{2}{\ve^{d/2+1}}\int\limits_{\mathbb{R}^d}
		J\left(\frac{|\xi-\tilde \xi|^2}{\ve}\right)
		\big(\rho_0(\xi)-\rho_0(\tilde\xi)\big)d\tilde \xi
			=-2 \int\limits_0^1	\int\limits_{\mathbb{R}^d} 
		J\left(| \zeta|^2\right) 	\partial^2_{\xi_j \xi_k}  \rho_0(\xi+\sqrt{\ve}t\zeta)\zeta_j\zeta_k
		d\zeta dt.
\end{equation*}
Then the sought boundedness follows from the facts that 
$\int\limits_{\mathbb{R}^d} J(|\zeta^2) |\zeta|^2  d\zeta< \infty$ and $\rho_0\in C_c^\infty(\mathbb{R}^d)$.

\end{proof}

\section{Eigenfunction concentration at a point on the boundary}
\label{SamaSamaSecttsiya}

This section is devoted to the proof of Theorem \ref{Th22}. Since the 
minimum $x_0=0$ of $ m(x)$ is now assumed to lie on the boundary, it is convenient to introduce local coordinates near the boundary.
For simplicity, we consider hereafter only the case $d=2$ while the argument in the case of higher dimensions is analogous (but much more cumbersome).


  Let $r$ be the signed distance from the point $x$ to the boundary $\partial \Omega$, positive when $x\in\Omega$, and $s$ be the length of the clockwise-oriented arc connecting the origin and the projection of $x$ on $\partial \Omega$. 
  Then, since the boundary is assumed to be $C^3$-smooth, the correspondence $x=X(r,s)$ is (locally) a $C^2$-diffeomorphism.
  
  In order to prove Theorem \ref{Th22} 
  we now consider the following functionals 
     \begin{equation}
     	\begin{aligned}
     		F_\ve(\rho_\ve)=\frac{1}{\ve^{13/6}}&\left(	\frac{1}{2\ve^2} \int\limits_{\mathbb{R}^2} J\left(\frac{|x-y|^2}{\ve^2}\right)(\rho_\ve(x)-\rho_\ve(y))^2dxdy \right.\\    
     		&\left. \quad + \int\limits_{\Omega}(m(x)-m(0)-\ve^{2/3}\Lambda_\ve+\ve)\rho^2_\ve(x)dx \right),
     		\label{funcF}
     	\end{aligned}
     \end{equation}
where $\Lambda_\ve$ is given by \eqref{asympt_Lambda}.
  The heuristics behind considering such an energy is as follows. 
  We propose the ansatz  for the first eigenfunctions in the form $\rho_\ve=\phi_\ve(r/\ve^{2/3})\psi_\ve(s/\ve^{1/2})$, with separated variables and two fine scales. 
 Substituting $\rho_\ve$ with frozen variable $s$ (corresponding to the  coarser scale) in \eqref{start_eq}--\eqref{main_operator_locper}  and keeping one term of the Taylor expansion $m(x)=m(0)+\alpha r+\dots$ we obtain the one-dimensional equation for $\phi_\ve$ and the corresponding eigenvalue $\Lambda_\ve$ that appears in \eqref{funcF} (see \eqref{eq_Lambda_ve}). The factor $\frac{1}{\ve^{13/6}}$ in \eqref{funcF} is chosen so that the Rayleigh quotient for $F_\ve$ matches the order of the last term in the expansion of $\lambda_\ve$. 
 
  We show the following compactness theorem. 
  \begin{thm}
  Assume that $F_\ve(\rho_\ve)\leq C$. Then 
\begin{itemize}  
	\item[(i)] as $\ve\to 0$,
 \begin{equation*}
   	\frac{1}{\ve^{7/6}}\left(
   	\int\limits_{\Omega_\ve} \rho_\ve^2 dx+	\int\limits_0^{\frac{1}{4}\ve^{5/8}}\int\limits_{|s|>\frac{1}{4}\ve^{11/24}} \rho_\ve^2 (X(r,s)) (1-r\kappa(s)) ds dr
   	   	\right)\to 0,
   \end{equation*}
where $\Omega_\ve=\{x\in\Omega: {\rm dist}(x,\partial\Omega)\geq \frac{1}{4}\ve^{5/8}\}$, $\kappa(s)$ is the curvature of $\partial \Omega $ at the point $X(0,s)$;
\item[(ii)] the restriction of $ \rho_\ve\left(X(\ve^{2/3}q,\ve^{1/2}t)\right)$ to the set  $0<q<\frac{1}{4}\ve^{-1/24}$, $-\frac{1}{4}\ve^{-1/24}<t<\frac{1}{4}\ve^{-1/24}$ continued by zero elsewhere, converge, up to a subsequence, 
strongly in $L^2(\mathbb{R}_+\times \mathbb{R})$ to a function of the form $\phi_0(q)\psi(t)$, where $\phi_0$ is the positive solution of \eqref{Airy_sln} with the Dirichlet condition $\phi_0(0)=0$ normalized by $\int\limits_0^\infty\phi_0^2(q)dq=1$, and $\psi\in H^1(\mathbb{R})$.
\end{itemize} 
\label{Th41}
  \end{thm}
  
  \begin{rem}
Combining statements $(i)$ and $(ii)$ 
of Theorem \ref{Th41}, we  get
 \begin{equation*}
	\frac{1}{\ve^{7/6}}\|\rho_\ve\|_{L^2(\Omega)} ^2\to
	\|\psi_0|_{L^2(\mathbb{R})}^2.
	 \end{equation*}
  \end{rem}
\begin{thm}
	\label{Th43}	
Functionals $F_\ve$ $\Gamma$-converge, with respect to the convergence defined in Theorem \ref{Th41}, to the functional $F_0$ given by
  \begin{equation*}
	F_0(\psi)=\frac{\theta}{2} \int\limits_{\mathbb{R}} (\psi^\prime (t))^2dt +  \int\limits_{\mathbb{R}}(1+bt^2)\psi^2(t)dt,
\end{equation*}
where 
\begin{equation*}
	b=\frac{1}{2}\left(\alpha\kappa(0)+\frac{\partial^2 m(0)}{\partial  x_2^2}\right),\quad \theta=\frac{\vartheta}{2},
\end{equation*}
  $\kappa(0)$ is the curvature of $\partial \Omega $ at the origin.
\end{thm}
  

\subsection{
Construction of a boundary layer solution}
\label{Consruct}
  We seek a first-order approximation of the first eigenfunction as a function of the single variable $r$. Then the kernel of the resulting one-dimensional integral operator writes as
\begin{equation}
		\begin{aligned}
	\frac{1}{\ve}\int\limits_{-L/2}^{L/2}  J\left(\frac{|X(r,s)-X(\tilde r,\tilde s)|^2}{\ve^2}\right) (1-\tilde r\kappa(\tilde s)) d \tilde s &  =I_0\Bigl(\frac{\tilde r-r}{\ve}\Bigr)\\ &
  +\frac{\kappa(s)(r-\tilde r)}{2(1-\tilde r\kappa(s))}
	I_0\Bigl(\frac{\tilde r-r}{\ve}\Bigr)+ \ve^2\mathcal{O_\ve}(r,\tilde r,s),
	\label{nu_nu}
\end{aligned}
\end{equation}
where $L$ is the length of the curve $\partial \Omega$,  
\begin{equation}
I_0(z)= \int\limits_{\mathbb{R}} J (s^2+z^2) \, ds, 
\label{I_0}
\end{equation}
and $\mathcal{O_\ve}(r,\tilde r, s)$ is a uniformly (in $\ve$) bounded function that equals zero if $|r-\tilde r|> C\ve$.

Indeed, since the function $J(z)$ has a compact support, we are interested only in those pairs $(r,s)$ and $(\tilde r,\tilde s)$ that satisfy $|X(r,s)-X(\tilde r,\tilde s)|\leq O(\ve)$,  
i.e.  $|r-\tilde r|\leq C\ve$, $|X(0,s)-X(0,\tilde s)|\leq C\ve$. Set $\tau =(\tilde s-s)/\ve$, 
$\gamma=(1-\tilde r\kappa(s))\tau$. Then we have
\begin{equation}
		\begin{aligned}
\frac{|X(r,s)-X(\tilde r,\tilde s)|^2}{\ve^2}
&=\tau^2 (1-\tilde r\kappa(s))^2 - 
\ve \tilde r \kappa'(s)(1-\tilde r\kappa(s))\tau^3+
\frac{(\tilde r-r)^2}{\ve^2}  \\ &+ (\tilde r-r)\tau^2\kappa(s)(1-\tilde r\kappa(s))+O(\ve^2);
\label{yadro}
	\end{aligned}
\end{equation}
\begin{equation}
1-\tilde r\kappa(\tilde s)=1-\tilde r\kappa(s)-\ve \tilde r  \kappa'(s)\tau +O(\ve^2),
\label{jacobian}
\end{equation}
and for sufficiently small $\ve$ we obtain
\begin{equation}
	\begin{aligned}
		\frac{1}{\ve}&\int\limits_{-L/2}^{L/2} J\Biggl(\frac{|X(r,s)-X(\tilde r,\tilde s)|^2}{\ve^2}\Biggr) (1-\tilde r\kappa(\tilde s)) d \tilde s \\ 
			&=I_0\Bigl(\frac{\tilde r-r}{\ve}\Bigr)
		+\kappa(s)\frac{\tilde r-r}{1-\tilde r\kappa(s)}
		\int\limits_{-\infty}^{\infty} J'\Bigl(\gamma^2+\frac{(\tilde r-r)^2}{\ve^2}\Bigr)\gamma^2
		\, d\gamma+\ve^2\mathcal{O_\ve}(r,\tilde r,s).
		\label{nu_nu3}
	\end{aligned}
\end{equation}
Integrating by parts in \eqref{nu_nu3} we arrive at \eqref{nu_nu}. The latter formula, in turn,  can be rewritten as  
	\begin{equation*}
		\begin{aligned}
		\frac{1}{\ve}&\int\limits_{-L/2}^{L/2} J\Biggl(\frac{|X(r,s)-X(\tilde r,\tilde s)|^2}{\ve^2}\Biggr) (1-\tilde r\kappa(\tilde s)) d \tilde s = 
I_\ve\Bigl(\frac{\tilde r-r}{\ve},\tilde r\Bigr) +\ve(|s|+\ve)\mathcal{\tilde O_\ve}(r,\tilde r,s),
	\end{aligned}
\end{equation*}
where 
$$I_\ve(z,\tilde r)=	I_0(z)-\frac{\ve\kappa(0)}{2}\frac{z}{1-\tilde r\kappa(0)}I_0(z).
$$

Next we perform the change of variables $\tilde r=\ve^{2/3} \tilde q$, $r=\ve^{2/3} q$ and construct an approximate solution of the following  equation
\begin{equation}
\begin{aligned}
		-\frac{1}{\ve}&\int\limits_{0}^{\ve^{-1/24}}
		I_\ve\Bigl(\frac{\tilde q-q}{\ve^{1/3}}, \ve^{2/3} \tilde q\Bigr)\phi_\ve(\tilde q)  d\tilde q \\  
		&+\frac{1}{\ve}\int\limits_{-\infty}^{\ve^{-1/24}} I_\ve\Bigl(\frac{\tilde q-q}{\ve^{1/3}}, \ve^{2/3} \tilde q\Bigr) d\tilde q \, \phi_\ve(q)+\alpha q \phi_\ve(q) =
		\Lambda_\ve \phi_\ve(q),\quad 0\leq q\leq \ve^{-1/24}.
		\label{eq_Lambda_ve}
	\end{aligned}
	\end{equation}
One can show that the principal eigenvalue $\Lambda_\ve$ of  the operator on the right-hand side of \eqref{eq_Lambda_ve}  exists and 
converges, as $\ve\to 0$, to  
the principal eigenvalue $\Lambda_0$ 
of the equation 	\eqref{Airy_sln}
with the boundary condition $\phi_0(0)=0$, where $\theta=\int_{-\infty}^{\infty}I_0(\zeta)\zeta^2d\zeta$. Also, the corresponding normalized positive eigenfunction of  \eqref{eq_Lambda_ve} converges 
to that of \eqref{Airy_sln}, $\phi_0(q)$. Moreover, $\phi_0$ is given by 
\begin{equation*}
\phi_0(q)=\frac{1}{\Pi}{\rm Ai}\bigl( (2\alpha/\theta)^{1/3} (q-\Lambda_0/\alpha)\bigr), \quad 
\Pi=\left(\int_0^\infty {\rm Ai}^2\bigl( (2\alpha/\theta)^{1/3} (q-\Lambda_0/\alpha)\bigr) dq\right)^{1/2},
\end{equation*}
where ${\rm Ai}(t)$ is the Airy function
\begin{equation*} 
{\rm Ai}(t)=\frac{1}{\pi}\int\limits_{0}^{\infty} \cos\Bigl(\frac{\rho^3}{3}+t\rho \Bigr)d\rho
\end{equation*}
that is a bounded solution of the equation 
\begin{equation*} 
-{\rm Ai}^{\prime\prime}(t)+t{\rm Ai}(t)=0, \quad t \in \mathbb{R},
\end{equation*}   
and $\Lambda_0$ is expressed via the first (maximal) zero $\eta_1$ of ${\rm Ai}(t)$  by
$\Lambda_0=-(\theta/2)^{1/3}\alpha^{2/3}\eta_1$. 

Notice that to find an appropriate approximation of the solution of \eqref{eq_Lambda_ve} near the right end $\ve^{-1/24}$ of the interval, the function $\phi_0(q)$ should be modified  to satisfy the zero Neumann boundary condition. 
However, to avoid the problem of construction of the boundary layer near  $\ve^{-1/24}$, we rather introduce a correction  $\phi_{0 \ve}(q)$ of $\phi_0(q)$ such that $\phi_{0 \ve}^{\prime}(\ve^{-1/24})>0$ leading 
to a supersolution on the right end. 
Namely, we consider 
\begin{equation*} 
	\phi_{0 \ve}(q)=\frac{1}{\Pi}{\rm Ai}\bigl( (2\alpha/\theta)^{1/3} (q-\Lambda_0/\alpha)\bigr)+\beta_\ve \chi(q-1)  {\rm Bi}\bigl( (2\alpha/\theta)^{1/3} (q-\Lambda_0/\alpha)\bigr),
\end{equation*} 
where ${\rm Bi}(t)$ is the Airy function of the second kind, 
\begin{equation*} 
	{\rm Bi}(x)=\frac{1}{\pi}\int\limits_{0}^{\infty}
	\Biggl[\exp\Bigl(-\frac{\rho^3}{3}+t\rho\Bigr)+ \sin\Bigl(\frac{\rho^3}{3}+t\rho\Bigr)\Biggr] d\rho,
\end{equation*}
\begin{equation*} 
	\beta_\ve= -\frac{2}{\Pi}\frac{{\rm Ai'}\bigl( (2\alpha/\theta)^{1/3} (\ve^{-1/24}-\Lambda_0/\alpha)\bigr)}{{\rm Bi'}\bigl( (2\alpha/\theta)^{1/3} (\ve^{-1/24}-\Lambda_0/\alpha)\bigr)}= \frac{1}{\Pi}\exp\Bigl(-\frac{4}{3}(2\alpha/\theta)^{1/2} (\ve^{-1/24}-\Lambda_0/\alpha)^{3/2}\Bigr)\bigl(1+o(1)\bigr),
	\end{equation*}  
and  $0\leq\chi(q)\leq 1$ is a smooth cut-off function such that 
\begin{equation*}
	\chi(q)=0 \ \text{for} \ q\leq 0 \ \text{and} \ \chi(q)=1 \ \text{for} \ q\geq 1.
\end{equation*}

Next we  construct an approximate solution of \eqref{eq_Lambda_ve}  (the pair $\phi_\ve(q)>0$ and $\Lambda_\ve$) such that
\begin{equation}
-\frac{1}{\ve}\int\limits_{0}^{\ve^{-1/24}} 
		I_\ve\Bigl(\frac{\tilde q-q}{\ve^{1/3}}, \ve^{2/3} \tilde q\Bigr)\phi_\ve(\tilde q)  d\tilde q 
		+\frac{1}{\ve}\int\limits_{-\infty}^{\ve^{-1/24}} I_\ve\Bigl(\frac{\tilde q-q}{\ve^{1/3}}, \ve^{2/3} \tilde q\Bigr) d\tilde q\, \phi_\ve(q)
		=\left(\Lambda_\ve-\alpha q+O(\ve^{1/2})\right)\phi_\ve(q)
\label{approx_solution}
\end{equation}
away from the right end $\ve^{-1/24}$, and 
\begin{equation*}
-\frac{1}{\ve}\int\limits_{0}^{\ve^{-1/24}} 
		I_\ve\Bigl(\frac{\tilde q-q}{\ve^{1/3}}, \ve^{2/3} \tilde q\Bigr)\phi_\ve(\tilde q)  d\tilde q 
		+\frac{1}{\ve}\int\limits_{-\infty}^{\ve^{-1/24}} I_\ve\Bigl(\frac{\tilde q-q}{\ve^{1/3}}, \ve^{2/3} \tilde q\Bigr) d\tilde q\, \phi_\ve(q) \geq \left(\Lambda_\ve-\alpha q+O(\ve^{1/2})\right)\phi_\ve(q)
\end{equation*}
on $[0, \ve^{-1/24}]$.
		
We seek an approximate eigenvalue  $\Lambda_\ve$  in the form 
\begin{equation}
\label{asympt_Lambda}
\Lambda_\ve=\Lambda_0-\alpha    q_\ve^*- \ve^{2/3}\hat q_1,
\end{equation}
with $\hat q_1$ defined below  in the construction of the asymptotic expansion away from the ends of the interval 
and $q_\ve^*=O(\ve^{1/3})$ to be specified in the analysis of the boundary layer at $0$. Note that the higher order terms in \eqref{asympt_Lambda}
are introduced to achieve the necessary accuracy  in \eqref{approx_solution} near zero,
and \eqref{asympt_Lambda} is a correct approximation of the principal eigenvalue at least to the order $O(\ve^{1/3})$.


We set $\phi_\ve(q):=\phi_{0\ve}(q+ q_\ve^*)+\ve^{2/3}\hat\phi_{0\ve}(q+ q_\ve^*)$  away from zero (with unknown for the moment function $ \hat\phi_{0\ve}$) 
and  observe that 
\begin{equation*}
	-\frac{\theta}{2}\phi_{0\ve}^{\prime\prime}(q+q_\ve^*) + \alpha \big(q+q_\ve^*\big) \phi_{0\ve}(q+q_\ve^*)=\Lambda_0 \phi_{0\ve}(q+q_\ve^*). 	
	\label{6.3}
\end{equation*}
Away from the ends $0$ and $\ve^{-1/24}$ of the interval 
we have 
$$
\int\limits_{0}^{\ve^{-1/24}} I_\ve\Bigl(\frac{\tilde q-q}{\ve^{1/3}}, \ve^{2/3} \tilde q\Bigr)\phi_\ve(\tilde q)  d\tilde q =\int\limits_{-\infty}^{\infty} I_\ve\Bigl(\frac{\tilde q-q}{\ve^{1/3}}, \ve^{2/3} \tilde q\Bigr)\phi_\ve(\tilde q)  d\tilde q,
$$
while
%
%
%
%
%
\begin{equation}
	\begin{aligned}
			-&\frac{1}{\ve}\int\limits_{-\infty}^{\infty}
	I_\ve\Bigl(\frac{\tilde q-q}{\ve^{1/3}}, \ve^{2/3} \tilde q\Bigr)\big(\phi_{0\ve}(\tilde q+q_\ve^*)-\phi_{0\ve}(q+q_\ve^*)\big)  d\tilde q 
	+\big(\alpha q-\Lambda_\ve \big) \phi_{0\ve}(q+q_\ve^*)\\ 
		=&-\frac{1}{\ve^{1/3}}\int\limits_{-\infty}^{\infty} I_0
		\bigl(\zeta\bigr)  
	\Big(\phi'_{0\ve}(q+q_\ve^*)\zeta+\frac{\ve^{1/3}}{2} \phi''_{0\ve}(q+q_\ve^*)\zeta^2  
+\frac{\ve^{2/3}}{6}\phi'''_{0\ve}(q+q_\ve^*)\zeta^3
 +\frac{\ve}{24}\phi^{(iv)}_{0\ve}(q+q_\ve^*)\zeta^4
 \Big)  d\zeta \\ 
 	&+\ve^{2/3}\frac{\kappa(0)}{2}\int\limits_{-\infty}^{\infty} I_0
\bigl(\zeta\bigr)  \frac{1}{1-\kappa(0)(\ve^{2/3}q-\ve\zeta)}
 \phi'_{0\ve}(q+q_\ve^*)\zeta^2  d\zeta  \\
 &+\ve^{1/3}\frac{\kappa(0)}{2}\int\limits_{-\infty}^{\infty} I_0
 \bigl(\zeta\bigr)  \frac{\zeta}{1-\kappa(0)(\ve^{2/3}q-\ve\zeta)} \int\limits_{0}^{\ve^{1/3}\zeta} \phi''_{0\ve}(q+\tilde q_\ve^*+\tau)(\ve^{1/3}\zeta-\tau)d\tau d\zeta\\ 
		&-\frac{1}{24\ve^{2/3}}\int\limits_{-\infty}^{\infty} I_0
		\bigl(\zeta\bigr)\int\limits_{0}^{\ve^{1/3}\zeta} \phi^{(v)}_{0\ve}(q+q_\ve^*+\tau)(\ve^{1/3}\zeta-\tau)^4d\tau d\zeta +\big(\alpha q-\Lambda_\ve \big) \phi_{0\ve}(q+q_\ve^*) \\ 
		=&	-\frac{1}{2}\phi^{\prime\prime}_{0\ve}(q+q_\ve^* )\int\limits_{-\infty}^{\infty}
		I_0\big(\zeta\big) \zeta^2 d\zeta \\ 
	&+\ve^{2/3}\left(\hat q_1\phi_{0\ve}(q+q_\ve^*)
		+ \frac{\kappa(0)}{2}\phi'_{0\ve}(q + q_\ve^*)\int\limits_{-\infty}^{\infty}
		I_0\big(\zeta\big) \zeta^2 d\zeta
		- \frac{1}{24}\phi^{(iv)}_{0\ve}(q + q_\ve^*)\int\limits_{-\infty}^{\infty}
		I_0\big(\zeta\big) \zeta^4 d\zeta\right)
		 \\
		&\quad+O(\ve)\Big(\phi''_{0\ve}(q+ q_\ve^*+\ve^{1/3}\zeta_\ve^*)+\phi^{(v)}_{0\ve}(q+ q_\ve^*+\ve^{1/3}\tilde\zeta_\ve^*)\Big)+\big(\alpha (q+q_\ve^*)-\Lambda_0 \big) \phi_{0\ve}(q+q_\ve^*)\\
		=&\ve^{2/3}\left(\hat q_1\phi_{0\ve}(q+q_\ve^*)
			+ \frac{\kappa(0)}{2}\phi'_{0\ve}(q + q_\ve^*)\int\limits_{-\infty}^{\infty}
		I_0\big(\zeta\big) \zeta^2 d\zeta- \frac{1}{24}\phi^{(iv)}_{0\ve}(q+q_\ve^*)\int\limits_{-\infty}^{\infty}
		I_0\big(\zeta\big) \zeta^4 d\zeta\right)
		\\
		&+O(\ve)\Big(\phi''_{0\ve}(q+ q_\ve^*+\ve^{1/3}\zeta_\ve^*)+\phi^{(v)}_{0\ve}(q+ q_\ve^*+\ve^{1/3}\tilde\zeta_\ve^*)\Big),
		\label{calculus}
\end{aligned}
\end{equation}
where we have set $\tilde q-q=:\ve^{1/3}\zeta$, and $\zeta_\ve^*$, $\tilde \zeta_\ve^*$ satisfy $|\zeta_\ve^*|<C$, $| \tilde \zeta_\ve^*|<C$. 
Similarly, 
\begin{equation*}
	\begin{aligned}
		-\frac{1}{\ve^{1/3}}&\int\limits_{-\infty}^{\infty}
		I_0\Bigl(\frac{\tilde q-q}{\ve^{1/3}}\Bigr)\big(\hat\phi_{0\ve}(\tilde q+q_\ve^*)-\hat\phi_{0\ve}(q+q_\ve^*)\big)  d\tilde q 
		+\ve^{2/3}\big(\alpha (\tilde q+q_\ve^*)-\Lambda_\ve \big) \hat\phi_{0\ve}(q+q_\ve^*)\\
			&=-\frac{\ve^{2/3}\theta}{2}\hat\phi^{\prime\prime}_{0\ve}(q+q_\ve^* ) +\ve^{2/3}\big(\alpha (q+q_\ve^*)-\Lambda_0 \big) \hat\phi_{0\ve}(q+q_\ve^*)+O(\ve).
	\end{aligned}
\end{equation*}
Note that \eqref{calculus} is satisfactory everywhere but near zero, where $\phi_0$ vanishes. To fix this we set  $\hat\phi_{0\ve}(q)=(1-\chi(q-1))\hat\phi_{0}(q+q_\ve^*)$ and, collecting terms of the order $\ve^{2/3}$, we obtain the following equation for 
$\hat\phi_0$,
\begin{equation}
\label{easy_zadachka} 
	\begin{aligned}
			&-\frac{\theta}{2}\hat\phi_0^{\prime\prime}(q)
	+\big(\alpha q-\Lambda_0 \big) \hat\phi_0(q)\\
	&=-\frac{\kappa(0)}{2}\phi'_0(q)\int\limits_{-\infty}^{\infty}
	I_0\big(\zeta\big) \zeta^2 d\zeta+\frac{1}{24}\phi^{(iv)}_0(q)\int\limits_{-\infty}^{\infty}
		I_0\big(\zeta\big) \zeta^4 d\zeta-\hat q_1 \phi_0(q), \quad q>0.
	\end{aligned}
\end{equation}
We supplement  \eqref{easy_zadachka} with the boundary condition $\hat\phi_0(0)=0$. Then the solvability condition for this problem yields
$$
\hat q_1=\frac{1}{24}\int_{-\infty}^{\infty}
I_0\big(\zeta\big) \zeta^4 d\zeta 
\int_{0}^{\infty}
\phi^{(iv)}_0(q)\phi_0(q)dq,
$$ 
and we choose the unique solution  $\hat\phi_0$ orthogonal to $\phi_0$.

Thus $\phi_\ve$ satisfies \eqref{approx_solution} away from the ends of the interval $0$  and $\ve^{-1/24}$.
%
%
%
%
%
Next we modify $\phi_\ve$ near zero,  
namely we postulate the ansatz 
$$\phi_\ve(q)=\phi_0\bigl(q+q_\ve^*\bigr)+\ve^{2/3}\hat\phi_0(q+ q_\ve^*)
+\ve^{1/3}\varphi_1(\xi) +\ve^{2/3}\varphi_2(\xi)+\ve\varphi_3(\xi)+\ve^{4/3}\varphi_4(\xi),
$$ 
$$
q_\ve^*=q_0\ve^{1/3}+q_1\ve^{2/3}+q_2\ve+q_3\ve^{4/3},
$$
where $\xi=\ve^{-1/3}q$, and $\varphi_k(\xi)\to 0$ as $\xi \to \infty$ faster than any power of $\xi$. Substitute this representation into the equation \eqref{eq_Lambda_ve} 
and choose the numbers  $q_0, \dots q_3$ and functions $\varphi_1,\dots \varphi_4$ so that
\begin{equation*}
	\begin{aligned}
		-&\frac{1}{\ve^{2/3}}\int\limits_{0}^{\infty}
		I_\ve \bigl(\tilde \xi-\xi, \ve^{2/3} q\bigr)
		\Big(\phi_0\bigl(\ve^{1/3}\tilde\xi+q_\ve^*
		\bigr)+\ve^{2/3}\hat\phi_0(\ve^{1/3}\tilde\xi+ q_\ve^*)+\sum_{k=1}^4\ve^{k/3}\varphi_k(\tilde \xi)
		\Big)  d\tilde\xi \\ 
		&+\frac{1}{\ve^{2/3}}\Big(
		\phi_0\bigl(\ve^{1/3}\xi+q_\ve^*
		\bigr)+\ve^{2/3}\hat\phi_0(\ve^{1/3}\xi+ q_\ve^*)
	                  +\sum_{k=1}^4\ve^{k/3}\varphi_k(\xi)
				\Big)\\ 
		&+\alpha\ve^{1/3}\xi 
		\Big(
		\phi_0\bigl(\ve^{1/3}\xi+q_\ve^*
		\bigr)+\ve^{2/3}\hat\phi_0(\ve^{1/3}\xi+ q_\ve^*)
		+\sum_{k=1}^4\ve^{k/3}\varphi_k(\xi)
		\Big)\\ 
		&-\Lambda_\ve \Big(
		\phi_0\bigl(\ve^{1/3}\xi+q_\ve^*
		\bigr)+\ve^{2/3}\hat\phi_0(\ve^{1/3}\xi+ q_\ve^*)
		+\sum_{k=1}^4\ve^{k/3}\varphi_k(\xi)
		\Big) \\ 
		&=O(\ve).
	\end{aligned}
\end{equation*}
To this end  we add and subtract $	-\ve^{-2/3}\int\limits_{-\infty}^{0}
I_\ve\bigl(\tilde \xi-\xi,\ve^{2/3} q\bigr) 
	\Big(\phi_0\bigl(\ve^{1/3}\tilde\xi+q_\ve^*
\bigr)+\ve^{2/3}\hat\phi_0(\ve^{1/3}\tilde\xi+ q_\ve^*)
\Big) d\tilde\xi $ and expand  $\phi_0$ in the Taylor series in the neighborhood of $0$. Collecting like terms  (in powers of $\ve^{1/3}$), we find the following equations for   $\varphi_k(\xi)$:
\begin{equation}
	\int\limits_{-\infty}^{0}
		I_0\bigl(\tilde \xi-\xi\bigr) 
		\phi'_0(0)(\tilde \xi+q_0) d\tilde \xi
		-\int\limits_{0}^{\infty} 	I_0\Bigl(\tilde \xi-\xi\Bigr) \varphi_1(\tilde \xi)d\tilde \xi	+\varphi_1(\xi)  = 0,
			\label{pershe}
\end{equation}
\begin{equation*}
	\int\limits_{-\infty}^{0}
	I_0\bigl(\tilde \xi-\xi\bigr) 
	\Bigl(\phi'_0(0)q_1+\frac{1}{2}\phi''_0(0)(\tilde \xi+q_0)^2  \Bigr) d\tilde \xi
	-\int\limits_{0}^{\infty} 		I_0\bigl(\tilde \xi-\xi\bigr) \varphi_2(\tilde \xi)d\tilde \xi	+\varphi_2(\xi)  = 0,
\end{equation*}
\begin{equation*}
	\begin{aligned}
	\int\limits_{-\infty}^{0}
	&	I_0\bigl(\tilde \xi-\xi\bigr) 
	\Bigl(\phi'_0(0) q_2+\phi''_0(0)(\tilde \xi+q_0)q_1  +\frac{1}{6}\phi'''_0(0)(\tilde \xi+q_0)^3+\hat\phi'_0(0)(\tilde \xi+q_0)\Bigr) d\tilde \xi\\ 
	&	-\int\limits_{0}^{\infty} 		I_0\bigl(\tilde \xi-\xi\bigr) \varphi_3(\tilde \xi)d\tilde \xi	+\varphi_3(\xi)  -\Lambda_0\varphi_1(\xi) = 0,
\end{aligned}
\end{equation*}
\begin{equation*}
	\begin{aligned}
			\int\limits_{-\infty}^{0}
			&I_0\bigl(\tilde \xi-\xi\bigr) 
		\Bigl(\phi'_0(0) q_3+\frac{1}{2}\phi''_0(0)\left(q_1^2+2(\tilde \xi+q_0)q_2 \right) +\frac{1}{2}\phi'''_0(0)(\tilde \xi+q_0)^2q_1+\frac{1}{24}\phi^{(iv)}_0(0)(\tilde \xi+q_0)^4\\
		&+\hat\phi'_0(0)q_1+\frac{1}{2}\hat\phi''_0(0)(\tilde \xi+q_0)^2\Bigr) d\tilde \xi-\frac{\kappa(0)}{2}\int\limits_{-\infty}^{0}(\tilde \xi-\xi)I_0\bigl(\tilde \xi-\xi\bigr) \left(\phi'_0(0)(\tilde \xi+q_0)+\varphi_1(\tilde\xi)\right)d\tilde\xi \\
			&-\int\limits_{0}^{\infty} 		I_0\bigl(\tilde \xi-\xi\bigr) \varphi_4(\tilde \xi)d\tilde \xi	+\varphi_4(\xi)  -\Lambda_0\varphi_2(\xi) +\alpha q_0\varphi_1(\xi) = 0.
	\end{aligned}
\end{equation*}
Consider equation \eqref{pershe}. Introducing  
 $\varphi(\xi):=\varphi_1(\xi)/  \phi'_0(0)$,  
we have
\begin{equation}
	-\int\limits_{0}^{\infty} 	I_0\big(\tilde \xi-\xi\big) \varphi(\tilde \xi)d\tilde \xi +\varphi(\xi)
=
	-\int\limits_{-\infty}^{0}
	I_0\big(\tilde \xi-\xi\big) 
	(\tilde\xi+q_0) d\tilde \xi.
	\label{eq_varphi}
\end{equation}
In Section \ref{BLS} we show that there exists $q_0>0$ such that \eqref{eq_varphi}
has a continuous solution $\varphi(\xi)$ decaying faster than any power of $\xi$ as
 $\xi\to\infty$. The numbers $q_1$, $q_2$, $q_3$ and functions $\varphi_2(\xi)$, $\varphi_3(\xi)$, $\varphi_4(\xi)$ are found analogously.

Finally we set, with a slight abuse of notation, 
\begin{equation}
		\begin{aligned}	
	\label{phi_ve_last}
	\phi_\ve(q)=\phi_{0 \ve}\bigl(q+q_\ve^*\bigr)&+\ve^{2/3}\hat\phi_{0 \ve}(q+ q_\ve^*)\\ 
&+\bigl(1-\chi(q-1)\bigr)\bigl(\ve^{1/3}\varphi_1(\xi) +\ve^{2/3}\varphi_2(\xi)+\ve\varphi_3(\xi)+\ve^{4/3}\varphi_4(\xi)\bigr).
	\end{aligned}	
	\end{equation} 
This function satisfies \eqref{approx_solution} away from the ends of the interval while for $q$ close to $\ve^{-1/24}$ ($0\leq \ve^{-1/24}-q\leq C\ve^{1/3}$), using  calculations in \eqref{calculus} we obtain

\begin{equation*}
		\begin{aligned}	
	\frac{1}{\ve}&\int\limits_{-\infty}^{\ve^{-1/24}} I_0\Bigl(\frac{\tilde q-q}{\ve^{1/3}}\Bigr) d\tilde q \phi_{0\ve}(q)
	-\frac{1}{\ve}\int\limits_{0}^{\ve^{-1/24}} I_0\Bigl(\frac{\tilde q-q}{\ve^{1/3}}\Bigr)\phi_{0\ve}(\tilde q) d\tilde q 
	+\alpha q \phi_{0\ve}(q) - 
	\Lambda_\ve \phi_{0\ve}(q)\\
&	= 
	\frac{1}{\ve}\int\limits_{-\infty}^{+\infty} I_0\Bigl(\frac{\tilde q-q}{\ve^{1/3}}\Bigr)\Bigl(\phi_{0\ve}(q)-\phi_{0\ve}(\tilde q) \Bigr) d\tilde q 
	+\alpha q \phi_{0\ve}(q) - 
	\Lambda_\ve \phi_{0\ve}(q)\\
&\quad\quad\quad	-	\frac{1}{\ve}\int\limits_{\ve^{-1/24}}^{+\infty} I_0\Bigl(\frac{\tilde q-q}{\ve^{1/3}}\Bigr)d\tilde q \  \phi'_{0\ve}(q_\ve^*)(q_\ve^*-q) 
\\
&\geq 
\frac{1}{\ve}\int\limits_{-\infty}^{+\infty} I_0\Bigl(\frac{\tilde q-q}{\ve^{1/3}}\Bigr)\Bigl(\phi_{0\ve}(q)-\phi_{0\ve}(\tilde q) \Bigr) d\tilde q 
+\alpha q \phi_{0\ve}(q) - 
\Lambda_\ve \phi_{0\ve}(q)\geq O(\ve^{1/2})\phi_{0\ve}(q),
		\end{aligned}
\end{equation*}
where $|q_\ve^*-\ve^{-1/24}|\leq  C\ve^{1/3}$.  

Notice that $\phi_{\ve}> c\ve^{1/3}$ in a neighborhood of zero with $c>0$ independent of $\ve$ (see Corollary \ref{cor72} and Remark \ref{rem73} ). Thus \eqref{approx_solution} also holds in a neighborhood of zero.  We summarise the properties of the function  $\phi_{\ve}$ in

\begin{lem} 
The function $\phi_{\ve}$ given by  \eqref{phi_ve_last} for $q\in [0,\ve^{-1/24}]$ (and set to be zero for $q<0$) satisfies 
\begin{equation}
	\begin{aligned}	
\frac{1}{\ve^2}\int_{-\infty}^{\ve^{-1/24}}\int_{\partial\Omega}  J\left(\frac{|X(\ve^{2/3}q,s)-X(\ve^{2/3}\tilde q,\tilde s)|^2}{\ve^2}\right)& j_\ve(\tilde q,\tilde s){\rm d} \tilde s\, \bigl(\phi_\ve(q)-\phi_\ve(\tilde q)\bigr)d\tilde q \\
&=(\Lambda_\ve -\alpha q)\phi_\ve(q)+R_\ve(q,s)\phi_\ve(q),
\label{formula_lemma51}
\end{aligned}
\end{equation}
where  
\begin{equation}
\label{j_ve}
j_\ve(\tilde q, \tilde s)=1-\ve^{2/3}\tilde q\kappa(\tilde s),
\end{equation}
and 
\begin{equation}
\label{Tse_Nashe_Vse}
\begin{aligned}
		&R_\ve(q,s)\geq -C(\ve^{1/2} +\ve^{1/3}|s|)\quad &\text{for}&\quad q\in [0,\ve^{-1/24}],\\
&R_\ve(q,s) \leq C(\ve^{1/2} +\ve^{1/3}|s|)\quad  &\text{for}&\quad q\in [0,\ve^{-1/24}-C \ve^{1/3}].
\end{aligned}
\end{equation}
\label{normal_factor}  
\end{lem}


\subsection{A one-dimensional equation on the half-line}
\label{BLS}
In this section we establish solvability of the equations appeared in  analysis of boundary layers 
in Section \ref{Consruct}. 

Consider the equation  
\begin{equation}
	-\int\limits_{0}^{\infty}
	I\big(\xi-\tilde\xi\big)v(\tilde\xi)d\tilde\xi+v(\xi)=f(\xi), \ \ \xi>0,
	\label{nu_nu4}
\end{equation}
 where
 \begin{equation}
 \label{trohy_umov}
 \begin{aligned}
& I, f \ \text{are continuous functions}, \ I(\xi)\geq 0,\ I(0)>0,\  I(\xi)=I(-\xi),\quad \int\limits_{-\infty}^{\infty} I(\xi)d\xi=1, \\
 &\exists R_0>0 \ \text{such that}\ I(\xi)=0 \ \text{if} \ |\xi|\geq R_0, \quad  f(\xi)\xi^n\to 0\ \text{as}\ \xi\to \infty, \forall n>0.
\end{aligned}
 \end{equation}
  We show that there exist a continuous solution $v(\xi)$ of this equation and a constant $C_v$ 
  such that $|v(\xi)-C_v|\xi^n\to 0$, $\forall n>0$.

{\it Step 1.} Consider solutions $v_N(\xi)$ of this equation on the finite interval $[0,N]$ for sufficiently large $N>0$:
\begin{equation*}
	-\int\limits_{0}^{N}
	I\big(\xi-\tilde\xi\big)v_N(\tilde\xi)d\tilde\xi+v_N(\xi)=f(\xi), \ \ 0\leq \xi\leq N.
\end{equation*}
The latter equation has a unique solution due to the Fredholm theorem. To show that functions $v_N(\xi)$ are uniformly bounded we construct the following barrier functions (supersolutions) 
\begin{equation*}
	W(\xi):=\frac{1}{R_0}-\frac{1}{2R_0+\xi},\ \xi\geq -R_0, \quad W(\xi):=0,  \ \xi<-R_0.
\end{equation*} 
Expanding  $W(\tilde\xi)=W(\xi)-W'(\xi)(\xi-\tilde\xi)+\frac{1}{2}W''(\xi)(\xi-\tilde\xi)^2-\frac{1}{6}W'''(\xi)(\xi-\tilde\xi)^3+\frac{1}{24}W^{(4)}(\xi^*)(\xi-\tilde\xi)^4$, we find
\begin{equation*}
\begin{aligned}
-\int\limits_{0}^{N}
	I\big(\xi-\tilde\xi\big)W(\tilde\xi)d\tilde\xi+W(\xi)&\geq \int\limits_{-\infty}^{\infty}I\big(\xi-\tilde\xi\big)\big(W(\xi)-W(\tilde\xi)\big)d\tilde\xi \\
	&= -\int\limits_{-\infty}^{\infty}I\big(\xi-\tilde\xi\big)\Big(\frac{1}{2}W''(\xi)(\xi-\tilde\xi)^2+\frac{1}{24}W^{(4)}(\xi^*)(\xi-\tilde\xi)^4\Big)d\tilde\xi\\ 
	&\geq 
	-\frac{1}{2}\, W''(\xi)\int\limits_{-\infty}^{\infty}I\big(\tau\big)\tau^2 d\tau=\frac{\theta}{(2R_0+\xi)^3},
	\label{nu_nu7}
\end{aligned}
\end{equation*}
where $\theta=\int\limits_{-\infty}^{\infty}I\big(\xi\big)\xi^2 d \xi$.
We can choose a  sufficiently large $k>0$ such that 
\begin{equation*}
	-\int\limits_{0}^{N}
	I\big(\xi-\tilde\xi\big)\Big(kW(\tilde\xi)-v_N(\tilde\xi)\Big)d\tilde\xi+kW(\xi)-v_N(\xi)\geq \frac{k\theta}{(2R_0+\xi)^3}-f(\xi)> 0, \ \ 0\leq \xi\leq N.
\end{equation*}
Then we have $kW(\xi)>v_N(\xi)$, due to the maximum principle \cite{LiCovWan2017}. 
Similarly,  $-kW(\xi)<v_N(\xi)$ for an appropriate finite number $k$. Notice that $k$ doesn't depend on $N$, therefore we can choose a sequence of numbers $N\to\infty$ such that $v_N(\xi)$ converge to a function $v(\xi)$ uniformly on compact subsets of $[0,\infty)$. Thus we have shown the existence of a bounded continuous solution $v$.

{\it Step 2.} It remains to show that $v(\xi)$ stabilizes to a constant as $\xi\to \infty$. To this end, observe that the function $v(\xi)$, continued  by zero for $\xi<0$, satisfies the equation
\begin{equation}
	-\int\limits_{-\infty}^{\infty}
	I\big(\xi-\tilde\xi\big)v(\tilde\xi)d\tilde\xi+v(\xi)=g(\xi), \ \ \xi\in \mathbb{R} ,
	\label{nu_nu9}
\end{equation}
with 
\begin{equation*}
	g(\xi)=
\begin{aligned}
\begin{cases} f(\xi) \quad \quad\quad\quad\quad\quad &\xi\geq 0; \\ 
		-\int\limits_{0}^{\infty}
		I\big(\xi-\tilde\xi\big)v(\tilde\xi)d\tilde\xi \quad &\xi< 0.
	\end{cases}
\end{aligned}
\end{equation*}
Applying the Fourier transform to \eqref{nu_nu9} we get
\begin{equation}
	-\hat I\big(\omega \big)\hat v(\omega)+\hat v(\omega)=\hat g(\omega), \quad \omega\in \mathbb{R}.
	\label{nu_nu11}
\end{equation}
The Fourier transforms $\hat I(\omega)= \int\limits_{-\infty}^{\infty} e^{-i \omega \xi} I(\xi)d \xi$ and $\hat g(\omega)$
are infinitely differentiable. These functions and their derivatives belong to $L^p(\mathbb{R})$, $\forall p>1$, and vanish at $\pm\infty$.
Also,  $\hat I\big(\omega \big)$ has the following properties: 
\begin{equation*}
\begin{aligned}
		\hat I\big(0 \big)=\int\limits_{-\infty}^{\infty} I\big(\xi\big)d\xi=1,  \ \text{and}\ 
	\hat I\big(\omega \big)=\int\limits_{-\infty}^{\infty}cos(\omega \xi)I\big(\xi\big)d\xi<1 \  \text{for} \ \omega \not =0, 
	\end{aligned}
\end{equation*}
and 
\begin{equation*}
	\hat I(\omega )=1-\frac{\omega^2}{2}\int\limits_{-\infty}^{\infty}z^2I(z)dz+
O(\omega^4)=1-\frac{\omega^2\theta}{2}+
O(\omega^4)
\end{equation*}
in a neighborhood of zero. Thus, resolving   \eqref{nu_nu11} with respect to $\hat v(\omega)$ and taking 
into account the fact that  $v(\xi)$ is a bounded function, we find
\begin{equation*}
\hat v(\omega)=\frac{1}{1-\hat I(\omega )}\hat g(\omega)+ C_1\delta(\omega), \ \ \omega\in \mathbb{R},
\end{equation*}
where $\delta(\omega)$ is the Dirac delta function centered at zero. Now represent $\hat v(\omega)$ as
\begin{equation*}
	\hat v(\omega)=\hat g(\omega) + U(\omega)\hat g(\omega)
	+  \hat I(\omega )\hat g(\omega)
	+\frac{2\hat g(\omega)}{\omega^2\theta}+C_1\delta(\omega), \ \ \omega\in \mathbb{R},
\end{equation*}
where
\begin{equation*}
U(\omega)=\frac{\hat I^2(\omega )}{1-\hat I(\omega )}-\frac{2}{\omega^2\theta}.
\end{equation*}
Then applying the inverse Fourier transform  we obtain 
\begin{equation}
 v(\xi)= g(\xi)+\frac{1}{2\pi}\int\limits_{-\infty}^{\infty} \bigl(U(\omega)\hat g(\omega)
	+  \hat I(\omega )\hat g(\omega)\bigr) e^{i\xi\omega} d\omega -\frac{1}{\theta}\int\limits_{-\infty}^{\infty}
|\xi-\tau| g(\tau)d\tau+\frac{C_1}{2\pi}.
\label{nu_vasche_kruto}
\end{equation}
(Recall that the inverse Fourier transform of $2/\omega^2$ is $-|\xi|$).


The function $U(\omega)\hat g(\omega)+ \hat I(\omega )\hat g(\omega)$ and all its derivatives belong to $L^1(\mathbb{R})$, therefore the second term in 
\eqref{nu_vasche_kruto} tends to zero as $\xi\to \pm\infty$ faster than any negative power of $|\xi|$. Since $v(\xi)=0$ for $\xi<0$, we can pass to the limit as $\xi\to -\infty $ in \eqref{nu_vasche_kruto} to find that $\int\limits_{-\infty}^{\infty} g(\tau) d\tau=0$ and $C_1=\frac{2\pi}{\theta}\int\limits_{-\infty}^{\infty}
\tau g(\tau)d\tau $. Thus $|v(\xi)-C_v|\xi^n\to 0$ as $\xi\to +\infty$, $\forall n>0$, 
with $C_v=\frac{2}{\theta}\int\limits_{-\infty}^{\infty} \tau g(\tau) d\tau$.

We summarize the obtained result in the following 
\begin{thm} Under conditions \eqref{trohy_umov} there exists a bounded continuous solution $v$ of problem \eqref{nu_nu4} and there is a constant $C_v$ such that $|v(\xi)-C_v|\xi^n\to 0$ as $\xi\to \infty$, $\forall n>0$. Moreover, if $f(\xi)\geq 0$ then $v(\xi)\geq 0$,  $C_v\geq 0$ (and $C_v> 0$ unless $f\equiv 0$).
\end{thm}
\begin{cor}
	\label{cor72} There exists a bounded continuous solution $\tilde\varphi (\xi)$ of the problem
	\begin{equation*}
		-\int\limits_{0}^{\infty}
		I_0\big(\xi-\tilde\xi\big)\tilde\varphi(\tilde\xi)d\tilde\xi+
	\tilde\varphi(\xi)=-\int\limits_{-\infty}^{0}I_0\big(\xi-\tilde\xi\big)\tilde\xi d\tilde\xi, \ \ \xi\geq 0
	\end{equation*}
(where $I_0$ is given by \eqref{I_0}). Moreover, it converges as $\xi\to\infty$ (faster than any negative power of $\xi$) to some constant $q_0>0$, 
\begin{equation*}
q_0=\lim_{\xi\to\infty}\tilde \varphi(\xi),
\end{equation*} 	
and the function $\xi+\tilde\varphi(\xi)$ is bounded away form zero on $[0,+\infty)$,
\begin{equation*}
\min_{\xi\in[0;\infty)}(\xi +\tilde \varphi(\xi))> 0.
\end{equation*}
\end{cor}
\begin{proof}
For every $\xi\geq 0$ we have
\begin{equation*}
	\begin{aligned}
		-\int\limits_{0}^{\infty}
	I_0\big(\tilde\xi-\xi\big) \big(\tilde\xi+\tilde \varphi(\tilde\xi)\big) d\tilde\xi+
\xi+\tilde\varphi(\xi)
		=-\int\limits_{-\infty}^{\infty} I_0\big(\tilde\xi-\xi\big)
		\tilde\xi d\tilde\xi+\xi=0.
	\end{aligned}
\end{equation*}
Therefore $U(\xi)=\xi+\tilde \varphi(\xi)$ satisfies
	\begin{equation}
	-\int\limits_{0}^{\infty}
	I_0\big(\xi-\tilde\xi\big)U(\tilde\xi)d\tilde\xi+
	U(\xi)=0, \ \ \xi\geq 0.
	\label{cor_p3}
\end{equation}

Let $\xi_m$ be a minimum point of $U(\xi)$ on $[0,\infty)$. We claim  that  $U(\xi_m)>0$.
Assume by contradiction that $U(\xi_m)\leq 0$.
From \eqref{cor_p3} we have 
\begin{equation*}
	\int\limits_{-\infty}^{\infty}
	I_0\big(\xi_m-\tilde\xi\big)U(\xi_m)d\tilde\xi=U(\xi_m)=\int\limits_{0}^{\infty}
	I_0\big(\xi_m-\tilde\xi\big)U(\tilde\xi)d\tilde\xi 
\end{equation*}
and, consequently,
\begin{equation}
	\int\limits_{-\infty}^{0}
	I_0\big(\xi_m-\tilde\xi\big)U(\xi_m)d\tilde\xi=\int\limits_{0}^{\infty}
	I_0\big(\xi_m-\tilde\xi\big)\Big(U(\tilde\xi)-U(\xi_m)\Big)d\tilde\xi. 	\label{cor_p4}
\end{equation}
Since  $I_0(z)\geq 0$,  the left-hand side of \eqref{cor_p4} is nonpositive while the right-hand side 
is nonnegative. Moreover, as $I_0(0)>0$, $U(\xi)=U(\xi_m)$ for all $\xi$ in a neighborhood of $\xi_m$ (within $[0,\infty)$). Therefore the set of minimum points of the function $U(\xi)$ is open (in the induced topology of $[0,\infty)$), and at the same time it is closed by the continuity of $U(\xi)$. Since $[0,\infty)$ is connected and $U(\xi)\to \infty $ as $\xi\to\infty$, the set of minimum points of $U(\tilde\xi)$  must be empty, a contradiction. Thus $U(\xi)\geq U(\xi_m)>0$.  
\end{proof}

 \begin{rem} 
 	\label{rem73} It is straightforward to see that the function $\varphi(\xi):=\tilde\varphi(\xi)-q_0$ is a continuous solution of \eqref{eq_varphi} decaying faster than any negative power of $\xi$, 
 	as $\xi\to\infty$. Then, since the function $\phi_1$ in the definition \eqref{phi_ve_last} of $\phi_\ve$ is given by $\phi_1(\xi)=\phi_0^\prime(0)\varphi(\xi)$, we have  	
$$
 	\phi_\ve(\ve^{1/3} \xi)=\ve^{1/3} \phi_0^\prime(0)(\xi+\tilde\phi(\xi))+O(\ve^{2/3})\geq c \ve^{1/3}, \quad c>0,
$$
when $0\leq \xi \leq C$.
 \end{rem}

\subsection{Factorization and representation for energy}
\label{Factor_Repr}

Consider the functional
\begin{equation*}
	\Phi_\ve(\rho_\ve)=	\frac{1}{2\ve^2} \int\limits_{\mathbb{R}^2\setminus \tilde\Omega_\ve} \int\limits_{\mathbb{R}^2\setminus \tilde\Omega_\ve} J\left(\frac{|x-y|^2}{\ve^2}\right)(\rho_\ve(x)-\rho_\ve(y))^2dxdy, 
\end{equation*}
where $\tilde\Omega_\ve=\{x\in\Omega: {\rm dist}(x,\partial\Omega)\geq \ve^{5/8}\}$.
Performing the change of variables $x=X(r,s)$ and $r=\ve^{2/3}q$, 
and setting
\begin{equation}
	\psi_\ve\left(\frac{s}{\ve^{1/2}}\right):=\frac{\rho_\ve(X(\ve^{2/3}q,s))}{ \phi_\ve(q)},\  \text{i. e.}\quad  \rho_\ve(X(\ve^{2/3}q,s))=\phi_\ve(q)\psi_\ve\left(\frac{s}{\ve^{1/2}}\right)
	\label{factorization}
\end{equation}
we have
\begin{equation*}
	\begin{aligned}
		\Phi_\ve(\rho_\ve)&=	\frac{1}{2\ve^{2/3}} \int\limits_{-\infty}^{\ve^{-1/24}} \int\limits_{-\infty}^{\ve^{-1/24}} 
		\int\limits_{\partial\Omega} \int\limits_{\partial\Omega}
		 \  J\left( 
		\frac{|X(\ve^{2/3} q,s)-X(\ve^{2/3}\tilde q,\tilde s)|^2}{\ve^2}\right) \\ 
		&\times\left(\phi_\ve(q) \psi_\ve\left(q,\frac{s}{\ve^{1/2}}\right) 
		-\phi_\ve(\tilde q) \psi_\ve\left(\tilde q,\frac{\tilde s}{\ve^{1/2}}\right)
		\right)^2 j_\ve(q,s)j_\ve(\tilde q,\tilde  s) dsd\tilde s dqd\tilde q
		\end{aligned} 
\end{equation*}
where $j_\ve$ is defined by \eqref{j_ve}. 
Then,
we find after some 
algebra\footnote{We use the following representation 
 \begin{equation*}
  	\begin{aligned}
&\Bigl( \phi_\ve(q) \psi_\ve(q,t)  -\phi_\ve(\tilde q) \psi_\ve( \tilde q,\tilde t) \Bigr)^2\\ 
&= \phi_\ve(q) \phi_\ve(\tilde q)
\left(\psi_\ve(q,t) -\psi_\ve( \tilde q,\tilde t)\right)^2 +
\left(\phi_\ve(q) - \phi_\ve(\tilde q)\right)
\left(\phi_\ve(q) \psi^2_\ve(q,t) -\phi_\ve(\tilde q) \psi^2_\ve( \tilde q,\tilde t)\right)
  \end{aligned} 
\end{equation*}
}

\begin{equation}
	\begin{aligned}
		\Phi_\ve&(\rho_\ve)=	\frac{1}{2\ve^{2/3}} \int\limits_{-\infty}^{\ve^{-1/24}} \int\limits_{-\infty}^{\ve^{-1/24}} 
			\int\limits_{\partial\Omega} \int\limits_{\partial\Omega}
		  J\left( 
		\frac{|X(\ve^{2/3} q,s)-X(\ve^{2/3}\tilde q,\tilde s)|^2}{\ve^2}\right) \\ 
		&\quad\times\left( \psi_\ve\left(q,\frac{s}{\ve^{1/2}}\right) 
		-\ \psi_\ve\left(\tilde q,\frac{\tilde s}{\ve^{1/2}}\right)
		\right)^2 \phi_\ve( q) j_\ve(q,s)\phi_\ve(\tilde q)j_\ve(\tilde q,\tilde  s) dsd\tilde s dqd\tilde q\\
		&+	\frac{1}{\ve^{2/3}} \int\limits_{-\infty}^{\ve^{-1/24}}  \int\limits_{\partial\Omega}   
		\phi_\ve( q) \psi_\ve^2\left(q,\frac{s}{\ve^{1/2}}\right)  j_\ve(q,s) \\
		&\quad\times
		\int\limits_{-\infty}^{\ve^{-1/24}} 
		 \Biggl( \int\limits_{\partial\Omega}
		 \  J\left( 
		\frac{|X(\ve^{2/3} q,s)-X(\ve^{2/3}\tilde q,\tilde s)|^2}{\ve^2}\right)  j_\ve(\tilde q,\tilde s) d\tilde s
	         \left( \phi_\ve(q)-\phi_\ve(\tilde q)\right)\Biggr) d\tilde q 
		 ds dq\\ 
	&	 =:\Phi_{1,\ve}(\rho_\ve)+\Phi_{2,\ve}(\rho_\ve),
		\end{aligned} 
	\label{Phi_ve}
\end{equation}
where, thanks to \eqref{formula_lemma51},  $\Phi_{2,\ve}$ can be written as
\begin{equation}
\begin{aligned}
\Phi_{2,\ve}(\rho_\ve)= \ve^{4/3}\int\limits_{-\infty}^{\ve^{-1/24}}  \int\limits_{\partial\Omega}    
		\big(\Lambda_\ve -\alpha q+R_\ve(q,s)\big)\phi^2_\ve(q) \psi_\ve^2\left(q,\frac{s}{\ve^{1/2}}\right)  j_\ve(q,s)
		 ds dq
\end{aligned} 
\label{Psi_ve}
\end{equation}
with $R_\ve(q,s)$ satisfying \eqref{Tse_Nashe_Vse}.

%
%
\subsection{Upper bounds}
\label{upperbound}

Given a function $\psi(t)\in C_c^\infty (\mathbb{R})$, we construct the recovery sequence by setting 
\begin{equation*}
	\rho_\ve\left(X(\ve^{2/3}q,\ve^{1/2}t)\right)=\phi_\ve(q)\psi(t)
	\end{equation*}
for $0\leq q\leq \ve^{-1/24}$. This formula defines $\rho_\ve$ in the near boundary layer $ \Omega\setminus\tilde\Omega_\ve$.  For $x\not\in \Omega\setminus\tilde\Omega_\ve$ we set $	\rho_\ve(x)=0$. 
Substitute  $\rho_\ve$ into the functional \eqref{funcF} 
to find, using  \eqref{yadro}-\eqref{jacobian} and the representation \eqref{Phi_ve}-\eqref{Psi_ve},
\begin{equation*}
\begin{aligned}
F_\ve&(\rho_\ve)=	
\frac{1}{2 \ve^{11/6}}
\int\limits_{0}^{\ve^{-1/24}} \int\limits_{0}^{\ve^{-1/24}} 
		\int\limits_{\mathbb{R}} \int\limits_{\mathbb{R}}
		J\Bigl(\frac{(\tilde q-q)^2}{\ve^{2/3}}+\frac{(\tilde t-t)^2}{\ve}
		\Bigr) 
	\left( \psi(\tilde t) 
		-\psi(t)
		\right)^2 \phi_\ve( q) \phi_\ve(\tilde q)
	dtd\tilde t dqd\tilde q\\
		&+
\frac{1}{\ve}
\int\limits_{0}^{\ve^{-1/24}} 
		\int\limits_{\mathbb{R}}
\left( m\left(X(\ve^{2/3}q,\ve^{1/2}t)\right)-m(0)-\alpha \ve^{2/3}q+\ve+\ve^{2/3}R_\ve (q,\ve^{1/2}t)\right) \psi^2(t)\phi^2_\ve(q)dtdq\\
&+o(1).
\end{aligned}
\end{equation*}
We set $\tau=(\tilde t-t)/\ve^{1/2}$ and using the Taylor expansions  $\psi(t+\ve^{1/2}\tau)- \psi(t)=\ve^{1/2}\tau \psi'(t)+\int\limits_0^{\ve^{1/2}\tau}\psi''(\ve^{1/2}\tau-\eta)d\eta$ and $m\left(X(\ve^{2/3}q,\ve^{1/2}t)\right)-m(0)= \alpha\ve^{2/3}q+b\ve q^2 +O(\ve^{4/3}(1+|q|^3))$, we obtain
\begin{equation*}
\begin{aligned}
F_\ve(\rho_\ve)=&	
\frac{1}{2 \ve^{1/3}} 
\int\limits_{0}^{\ve^{-1/24}} \int\limits_{0}^{\ve^{-1/24}} 
		\int\limits_{\mathbb{R}} \int\limits_{\mathbb{R}} J\Bigl(\frac{(\tilde q-q)^2}{\ve^{2/3}}+\tau^2
		\Bigr) \left(\tau \psi^\prime (t)\right)^2 d\tau  dt \phi_\ve( q) \phi_\ve(\tilde q) dq d\tilde q
\\
		&+
\frac{1}{\ve}
\int\limits_{0}^{\ve^{-1/24}} 
		\int\limits_{\mathbb{R}}
\left( (1+bt^2)\ve
+ \ve^{2/3}R_\ve (q,\ve^{1/2}t)\right) \psi^2(t)\phi^2_\ve(q)dtdq+ +o(1).
\end{aligned}
\end{equation*}
Thus we have proved 
\begin{lem}
	\label{Lemma71}
For any $\psi\in C_c^\infty(\mathbb{R})$, 
\begin{equation*}
\begin{aligned}
F_\ve(\rho_\ve)=
\frac{\theta}{2} \int\limits_{\mathbb{R}} (\psi^\prime (t))^2  dt +
\int\limits_{\mathbb{R}}(1+bt^2) \psi^2(t)dt+o(1)
\end{aligned}
\end{equation*}
with 	$ \rho_\ve(X(\ve^{2/3}q, \ve^{1/2}t))=\phi_\ve(q)\psi(t)$ in $\tilde\Omega_\ve$
and  $ \rho_\ve=0$ otherwise.
\end{lem}

The following result allows one to pass to the dense set 
$C_c^\infty(\mathbb{R})$ when checking $\Gamma$-$\limsup$ and $\Gamma$-$\liminf$ inequalities.

\begin{lem} $\forall \,\psi \in H^1(\mathbb{R})$, let $\rho_\ve$ be given by 
$ \rho_\ve(X(\ve^{2/3}q, \ve^{1/2}t))=\phi_\ve(q)\psi_\ve(t)$ in $\tilde\Omega_\ve$, where $\psi_\ve(t)=\psi(t)\left(1-\chi\big(t-\frac{L}{2\ve^{1/2}}\big)-\chi\big(-t-\frac{L}
	{2\ve^{1/2}}\big)\right)$,
	 and  $ \rho_\ve=0$ otherwise. Then 
\begin{equation*}
F_\ve(\rho_\ve)\leq C\left(\int\limits_{\mathbb{R}} (\psi_\ve^\prime (t))^2  dt +\int\limits_{\mathbb{R}} \psi_\ve^2(t)(1+t^2) dt\right).
\end{equation*} 
\label{Lemma72}
\end{lem}
\begin{proof}
We have 
\begin{equation}
\begin{aligned}
	\label{77}
\big|\psi_\ve(\tilde t)- \psi_\ve(t)\big|^2=
\Biggl(\int\limits_{t}^{\tilde t} \psi_\ve^\prime(\tau) d\tau\Biggr)^2 \leq 
|\tilde t-t| \left|\int\limits_{t}^{\tilde t}\Bigl(\psi_\ve^{\prime} (\tau)\Bigr)^2 d\tau \right|
\end{aligned}
\end{equation}
Notice also that there exists $T>0$ such that (for sufficiently small $\ve$)
$$
J\Biggl(
\frac{|X(\ve^{2/3}q,\ve^{1/2}t)-X(\ve^{2/3}\tilde q,\ve^{1/2}\tilde t)|^2}{\ve^2}\Biggr)=0
$$
when $|q-\tilde q|>T\ve^{1/3}$, $|t-\tilde t|>T\ve^{1/2}$.     
Hence
\begin{equation*}
	\begin{aligned}
		\Phi_{1,\ve}(\psi_{\ve}(t))\leq
		C \frac{Q_\ve}{\ve^{11/6}} \int\limits_{0}^{\ve^{-1/24}} \phi_\ve(q) \int\limits_{\tilde q_\ve^-(q)}^{\tilde q_\ve^+(q)}\, 
			 \phi_\ve(\tilde q) \, d\tilde q dq,
	\end{aligned}
\end{equation*}
where $\tilde q_\ve^-(q)=\max\{0,q-T\ve^{1/3}\}$, $\tilde q_\ve^+(q)=\min\{\ve^{-1/24},q+T\ve^{1/3}\}$, and
\begin{equation*}
		Q_\ve=
		\int\limits_{-\infty}^{\infty}
		\int\limits_{t-T\ve^{1/2}}^{t+T\ve^{1/2}}\,   \Bigl(\psi_{\ve}(t)- \psi_{\ve}(\tilde t)\Bigr)^2d\tilde t dt.
	\end{equation*}
From \eqref{77}, integrating by parts we find  that
\begin{equation*}
	\int\limits_{t-T\ve^{1/2}}^{t+T\ve^{1/2}}\,   \Bigl(\psi_{\ve}(t)- \psi_{\ve}(\tilde t)\Bigr)^2d\tilde t \leq \int\limits_{t-T\ve^{1/2}}^{t+T\ve^{1/2}}  
	(\tilde t-t)	\int\limits_{t}^{\tilde t}\Bigl(\psi_\ve^{\prime} (\tau)\Bigr)^2 d\tau
		 d\tilde t  \leq \, \frac{T^2\ve}{2}
	\int\limits_{t-T\ve^{1/2}}^{t+T\ve^{1/2}} \Bigl(\psi_\ve^{\prime} (\tau)\Bigr)^2 d \tau.
\end{equation*}
Therefore
\begin{equation*}
 		Q_\ve
 \leq 	\frac{T^2\ve}{2}\int\limits_{-\infty}^{\infty}
 \int\limits_{t-T\ve^{1/2}}^{t+T\ve^{1/2}} \bigl(\psi^{\prime}_{\ve}(\tau)\bigr)^2 d \tau dt=
 \frac{T^2\ve}{2}\int\limits_{-\infty}^{\infty}
 \int\limits_{-T\ve^{1/2}}^{T\ve^{1/2}} \bigl(\psi^{\prime}_{\ve}(t+h)\bigr)^2 dhd t= T^3 \ve^{3/2} \int\limits_{-\infty}^{\infty}
  \Bigl(\psi^{\prime}_{\ve}(\tau)\Bigr)^2 d\tau.
\end{equation*}
By construction of the function $\phi_\ve$, it holds for sufficiently small $\ve$ that $\phi_\ve (\tilde q)\leq C\phi_\ve (q)$ for $ \tilde q_\ve^-(q)\leq \tilde q \leq \tilde q_\ve^+(q)$ with $C>0$ 
independent of $q$ and $\ve$.
 Thus,
 \begin{equation*}
	\begin{aligned}
			\Phi_{1, \ve}(\psi_{\ve}(t)) 
\leq C_1
\int\limits_{0}^{\ve^{-1/24}}
	\phi^2_\ve(q)dq \int\limits_{-\infty}^{\infty}	\Bigl(\psi^{\prime}_{\ve}(\tau)\Bigr)^2 d\tau.
			\end{aligned}
\end{equation*}
Also, in view of \eqref{Tse_Nashe_Vse} and \eqref{Psi_ve} the remaining terms can be estimated as follows
\begin{equation*}
	\begin{aligned}
		\frac{1}{\ve^{13/6}}\left(	\Phi_{2,\ve}(\rho_\ve)	+  \int\limits_{\Omega}(m(x)-m(0)-\ve^{2/3}\Lambda_\ve+\ve)\rho^2_\ve(x)dx \right)\\ 
		\leq
		\int\limits_{0}^{\ve^{-1/24}}
			\int\limits_{\mathbb{R}}
		(1+bt^2)\psi_\ve^2(t)\phi^2_\ve(q)dtdq +O(\ve^{1/6}).
			\end{aligned} 
\end{equation*}
Lemma is proved.
\end{proof}

\subsection{Proof of Theorem \ref{Th41} (compactness)}

\label{CompactThM}

We consider a sequence of functions $\rho_\ve$ such that $F_\ve(\rho_\ve)\leq C$. Since   $m(x)-m(0)\geq c \, \ve^{5/8}$ (with $c>0$) in $\Omega_\ve=\{x\in\Omega: {\rm dist}(x,\partial\Omega)\geq \frac{1}{4}\ve^{5/8}\}$ we have
 \begin{equation*} 
 	\frac{1}{\ve^{7/6}}\int\limits_{\Omega_\ve}
 	\rho^2_\ve(x)dx\leq C \ve^{3/8}.
 	\end{equation*}
Similarly, as $m(X(r,s))-m(0)\geq c \, \ve^{11/12}$  for $0<r<\frac{1}{4}\ve^{5/8}$, $|s|>\frac{1}{4}\ve^{11/24}$, 
\begin{equation*} 
	\frac{1}{\ve^{7/6}} \int\limits_0^{\frac{1}{4}\ve^{5/8}}\int\limits_{|s|>\frac{1}{4}\ve^{11/24}} \rho_\ve^2 (X(r,s)) (1-r\kappa(s)) ds dr
	\leq C \ve^{1/12}.
\end{equation*}
This proves item $(i)$ of Theorem \ref{Th41}.
  
Now represent the restriction of $\rho_\ve$ to $\Omega\setminus \tilde\Omega_\ve$ in the form \eqref{factorization}. Then thanks to \eqref{Phi_ve} and  \eqref{Psi_ve} we have
 \begin{equation*}
F_\ve(\rho_\ve)\geq\Psi_\ve(\psi_\ve) +\frac{c}{\ve^{37/24}}\int\limits_{\tilde\Omega_\ve}
\rho^2_\ve(x)dx + \frac{c}{\ve^{1/12}}\int\limits_0^{\ve^{-1/24}} \int\limits_{|t|>\ve^{-1/24}} \phi^2_\ve(q)\psi^2_\ve(q,t)dtdq,
\end{equation*}
where 
\begin{equation}
\begin{aligned}
	&\Psi_\ve(\psi_\ve) =
		\frac{1}{2\ve^{11/6}} \int\limits_{0}^{\ve^{-1/24}} \int\limits_{0}^{\ve^{-1/24}} 
	\int\limits_{-\ve^{-1/24}}^{\ve^{-1/24}}   \int\limits_{-\ve^{-1/24}}^{\ve^{-1/24}} 
	J\left( 
	\frac{|X(\ve^{2/3} q,\ve^{1/2}t)-X(\ve^{2/3}\tilde q,\ve^{1/2}\tilde t)|^2}{\ve^2}\right)\\ 
	&\quad\quad\times\left( \psi_\ve\left(q,t\right) 
	-\ \psi_\ve\left(\tilde q,\tilde t \right)
	\right)^2 \phi_\ve( q) j_\ve(q,\ve^{1/2} t)\phi_\ve(\tilde q)j_\ve(\tilde q,\ve^{1/2}\tilde t) dtd\tilde t dqd\tilde q\\ 
	&+\frac{1}{\ve}
	\int\limits_{0}^{\ve^{-1/24}} 
		\int\limits_{-\ve^{-1/24}}^{\ve^{-1/24}}
	\left(m\big(X(\ve^{2/3} q,\ve^{1/2}t)\big)-m(0) -\alpha\ve^{2/3} q +\ve+
	\ve^{2/3} R_\ve(q,\ve^{1/2}t)\right)\\
	&\quad\quad	\times\phi^2_\ve(q)\psi^2_\ve(q,t)
	 j_\ve(q,\ve^{1/2} t)dtdq.
	 \label{8.2}
 	\end{aligned}
\end{equation}
Observe that, in view of \eqref{Tse_Nashe_Vse}, we have the following lower bound for $0\leq q\leq \ve^{-1/24}$ and $-\ve^{-1/24}\leq t\leq \ve^{-1/24}$.
\begin{equation*}
m\big(X(\ve^{2/3} q,\ve^{1/2}t)\big)-m(0) -\alpha\ve^{2/3} q +\ve+
\ve^{2/3} R_\ve(q,\ve^{1/2}t)\geq \frac{\ve}{2}bt^2+\ve+O\big(\ve^{13/12}\big).
\end{equation*}
Therefore both terms in \eqref{8.2} are nonnegative.
Thus
 \begin{equation}
\Psi_\ve(\psi_\ve) \leq F_\ve(\rho_\ve)\leq C    \quad \text{and} \quad
			\int\limits_{0}^{\ve^{-1/24}}
			\int\limits_{-\ve^{-1/24}}^{\ve^{-1/24}}
(t^2+1)\phi^2_\ve(q)\psi^2_\ve(q,t)
j_\ve(q,\ve^{1/2} t)dtdq<C.
\label{boundt2}
\end{equation}
Since $J$ is a continuos function and $J(0)>0$, there exist $q^*>0$ and $t^*>0$ such that 
 \begin{equation*}
 	J\left( 
\frac{|X(\ve^{2/3}q,\ve^{1/2}  t)-X(\ve^{2/3}(q+h),\ve^{1/2}(t+l))|^2}{\ve^2}\right)
\geq c>0
\end{equation*}
for $0\leq h<q^*\ve^{1/3}$, $-t^*\ve^{1/2}< l<t^*\ve^{1/2}$.
Hence 
\begin{equation*}
	 \int\limits_{0}^{\frac{1}{2}\ve^{-1/24}} 
		\int\limits_{-\frac{1}{2}\ve^{-1/24}}^{\frac{1}{2}\ve^{-1/24}} \int\limits_{0}^{q^*\ve^{1/3}}    \int\limits_{-t^*\ve^{1/2}}^{t^*\ve^{1/2}} 
\left( \psi_\ve\left(q,t\right) 
		-\ \psi_\ve\left(q+h,t+l \right)
		\right)^2 \phi_\ve( q) \phi_\ve(q+h) dldh dtdq	\leq C \ve^{11/6}.
	\end{equation*}

Since $\phi_\ve(q+h)\geq c\, \phi_\ve(q)$ for $0\leq h<q^*\ve^{1/3}$ and with $c>0$ independent of $q$ and $\ve$,
\begin{equation}
	\int\limits_{0}^{\frac{1}{2}\ve^{-1/24}} 
	\int\limits_{-\frac{1}{2}\ve^{-1/24}}^{\frac{1}{2}\ve^{-1/24}} \int\limits_{0}^{q^*\ve^{1/3}}    \int\limits_{-t^*\ve^{1/2}}^{t^*\ve^{1/2}} 
	\left( \psi_\ve\left(q,t\right) 
	-\ \psi_\ve\left(q+h,t+l \right)
	\right)^2 \phi^2_\ve( q) dldh dtdq	\leq C \ve^{11/6}.
	\label{8.7}
\end{equation}

There exist $h_\ve^+\in\bigl[\frac{3}{4} q^* \ve^{1/3}, q^* \ve^{1/3}\bigr]$ and
$h_\ve^-\in\bigl[0,\frac{1}{4} q^* \ve^{1/3}\bigr]$ such that
\begin{equation}
   \	\int\limits_{0}^{\frac{1}{2}\ve^{-1/24}} 
   \int\limits_{-\frac{1}{2}\ve^{-1/24}}^{\frac{1}{2}\ve^{-1/24}}
	\int\limits_{-t^*\ve^{1/2}}^{t^*\ve^{1/2}}
	\, \phi^2_\ve(q) \Bigl(\psi_\ve(q,t)-\psi_\ve(q+h_\ve^{\pm},t+l)\Bigr)^2 dldtdq\leq C\ve^{3/2},
	\label{8.8}
	\end{equation}
and there exist $l_\ve^+\in\bigl[\frac{3}{4} t^* \ve^{1/2}, t^* \ve^{1/2}\bigr]$ and
$l_\ve^-\in\bigl[0,\frac{1}{4} t^* \ve^{1/2}\bigr]$ such that
\begin{equation}
		\int\limits_{0}^{\frac{1}{2}\ve^{-1/24}} 
	\int\limits_{-\frac{1}{2}\ve^{-1/24}}^{\frac{1}{2}\ve^{-1/24}}
\int\limits_{0}^{q^*\ve^{1/3}}
	\, \phi^2_\ve(q) \Bigl(\psi_\ve(q,t)-\psi_\ve(q+h,t+l_\ve^{\pm})\Bigr)^2 dhdtdq\leq C\ve^{4/3}.
	\label{8.9}
\end{equation}

 Consider
 \begin{equation*}
	\begin{aligned}
	\bar\psi_\ve(q,t)=	
	\frac{1}{\big(h_\ve^+-h_\ve^-\big)\big(l_\ve^+-l_\ve^-\big)}
		\int\limits_{h_\ve^-}^{h_\ve^+} \, 	\int\limits_{l_\ve^-}^{l_\ve^+}\psi_\ve(h+q,l+t)dhdl.  	
	\end{aligned}
\end{equation*}
Applying Jensen's inequality, from \eqref{8.7} we derive
\begin{equation}
		\int\limits_{0}^{\frac{1}{2}\ve^{-1/24}} 
	\int\limits_{-\frac{1}{2}\ve^{-1/24}}^{\frac{1}{2}\ve^{-1/24}}
  \phi^2_\ve(q) \Bigl(\psi_\ve(q,t)-\bar\psi_\ve(q,t)\Bigr)^2 dtdq\leq C\ve^{2/3}.
  \label{phipsi}
\end{equation}
Further we calculate 
  \begin{equation*}
 	\begin{aligned}
 &	\frac{\partial \bar\psi_\ve}{\partial q}	=	\frac{1}{\big(h_\ve^+-h_\ve^-\big)\big(l_\ve^+-l_\ve^-\big)}
 		\int\limits_{l_\ve^-}^{l_\ve^+}
 		\Bigl(\psi_\ve(h_\ve^++q,l+t)-\psi_\ve(h_\ve^-+q,l+t)\Bigr)dl\\
 		=&\frac{1}{\big(h_\ve^+-h_\ve^-\big)\big(l_\ve^+-l_\ve^-\big)}
 \Biggl[\int\limits_{l_\ve^-}^{l_\ve^+}\Bigl(\psi_\ve(h_\ve^++q,l+t)-\psi_\ve(q,t)\Bigr)dl
 +\int\limits_{l_\ve^-}^{l_\ve^+}\Bigl(\psi_\ve(q,t)-\psi_\ve(h_\ve^-+q,l+t)\Bigr)dl\Biggr]
 	 	\end{aligned}
 \end{equation*}
 and using Jensen's inequality we get
  \begin{equation*}
 	\begin{aligned}
 	\Biggl|\frac{\partial \bar\psi_\ve}{\partial q}	\Biggr|^2\leq&	\frac{2}{\big(h_\ve^+-h_\ve^-\big)^2\big(l_\ve^+-l_\ve^-\big)^2}
 \Biggl[\int\limits_{l_\ve^-}^{l_\ve^+}\Bigl(\psi_\ve(h_\ve^++q,l+t)-\psi_\ve(q,t)\Bigr)dl\Biggr]^2\\ 
 		+&\frac{2}{\big(h_\ve^+-h_\ve^-\big)^2\big(l_\ve^+-l_\ve^-\big)^2}
 \Biggl[\int\limits_{l_\ve^-}^{l_\ve^+}\Bigl(\psi_\ve(q,t)-\psi_\ve(h_\ve^-+q,l+t)\Bigr)dl\Biggr]^2 \\ 
 \leq&
 	\frac{2}{\big(h_\ve^+-h_\ve^-\big)^2\big(l_\ve^+-l_\ve^-\big)}
 \int\limits_{l_\ve^-}^{l_\ve^+}\Bigl(\psi_\ve(h_\ve^++q,l+t)-\psi_\ve(q,t)\Bigr)^2dl\\ 
 +&\frac{2}{\big(h_\ve^+-h_\ve^-\big)^2\big(l_\ve^+-l_\ve^-\big)}
 \int\limits_{l_\ve^-}^{l_\ve^+}\Bigl(\psi_\ve(q,t)-\psi_\ve(h_\ve^-+q,l+t)\Bigr)^2dl.  
 	\end{aligned}
 \end{equation*}
Therefore by virtue of \eqref{8.8},
\begin{equation}
	\int\limits_{0}^{\frac{1}{2}\ve^{-1/24}} 
\int\limits_{-\frac{1}{2}\ve^{-1/24}}^{\frac{1}{2}\ve^{-1/24}}
\, \phi^2_\ve(q)  \,	\Biggl|\frac{\partial \bar\psi_\ve}{\partial q}	\Biggr|^2dtdq\leq
C\ve^{1/3}.
\label{boundq}
\end{equation}
 Similarly,
\begin{equation*}
 \begin{aligned}
 	\Biggl|\frac{\partial \bar\psi_\ve}{\partial t}	\Biggr|^2\leq&
 	\frac{2}{\big(h_\ve^+-h_\ve^-\big)\big(l_\ve^+-l_\ve^-\big)^2}
 	\int\limits_{h_\ve^-}^{h_\ve^+}\Bigl(\psi_\ve(q,l_\ve^++t)-\psi_\ve(q,t)\Bigr)^2dh\\ 
 	+&\frac{2}{\big(h_\ve^+-h_\ve^-\big)\big(l_\ve^+-l_\ve^-\big)^2}
 	\int\limits_{h_\ve^-}^{h_\ve^+}\Bigl(\psi_\ve(q,t)-\psi_\ve(q,l_\ve^-+t)\Bigr)^2dh,  
 \end{aligned}
\end{equation*}
 and consequently, due to \eqref{8.9},
 \begin{equation}
 	\int\limits_{0}^{\frac{1}{2}\ve^{-1/24}} 
 \int\limits_{-\frac{1}{2}\ve^{-1/24}}^{\frac{1}{2}\ve^{-1/24}}
 	\, \phi^2_\ve(q)  \,	\Biggl|\frac{\partial \bar\psi_\ve}{\partial t}	\Biggr|^2  dt dq \leq
 	C.
 	\label{boundt}
 \end{equation}

\begin{prop} For any 
	function $u\in H^1\big([0,+\infty)\big)$ with compact support it holds that
	\begin{equation}
		\begin{aligned}
			\frac{1}{2}\int\limits_0^\infty &\left(\phi^\prime_0(q)u(q)\right)^2 dq + \frac{2\alpha}{\theta} \int\limits_0^\infty \phi^2_0(q) u^2(q)\, qdq\\
			&\leq 2\left(1+\frac{\Lambda_0}{\theta}\right)\int\limits_0^\infty \phi^2_0(q) u^2(q)dq +2 \int\limits_0^\infty \left(\phi_0(q) u^\prime(q)\right)^2 dq.
			\label{8.17}
		\end{aligned}
	\end{equation} 
\label{Prop81}
\end{prop}
\begin{proof}
	Integrating by parts and using equation \eqref{Airy_sln} we get 
	\begin{equation}
		\begin{aligned}
			\int\limits_0^\infty \left(\phi'_0(q) u(q)\right)^2dq =
			\frac{2}{\theta}\int\limits_0^\infty\phi^2_0(q)u^2(q) \left(\Lambda_0 -\alpha q\right)dq
			-2		\int\limits_0^\infty \phi^\prime_0(q) u(q) \phi_0(q)u^\prime(q) dq.
			\label{8.18}
		\end{aligned}
	\end{equation} 
	Then, applying the Cauchy-Schwarz inequality we estimate the second term as follows 
	\begin{equation}
		\begin{aligned}
			-2&		\int\limits_0^\infty \phi^\prime_0(q) u(q) \phi_0(q)u^\prime(q)  dq \leq
			2\left(	\int\limits_0^\infty \left(\phi^\prime_0(q)u(q)\right)^2dq
			\int\limits_0^\infty \left( \phi_0(q) u^\prime(q)\right)^2 dq\right)^{1/2}\\ 
			&\leq \frac{1}{2} \int\limits_0^\infty \left( \phi^\prime_0(q)u(q)\right)^2dq+
			2\int\limits_0^\infty \left( \phi_0(q)u^\prime(q)\right)^2 dq.
			\label{8.19}
		\end{aligned}
	\end{equation}
	Combining \eqref{8.18} and \eqref{8.19} we obtain \eqref{8.17}.
\end{proof}

To prove the compactness of the sequence of functions  $\rho_\ve(X(\ve^{2/3}q,\ve^{1/2}t))= \phi_\ve(q)\psi_\ve(q,t)$ (restricted to $[0,\frac{1}{4}\ve^{-1/24}]\times [-\frac{1}{2}\ve^{-1/24},\frac{1}{2}\ve^{-1/24}]$ and extended by zero elsewhere), 
consider 
\begin{equation*}
	\begin{aligned}
	w_\ve&(q,t)=\, \phi_\ve(q)\bar\psi_\ve(q,t)\big(1-\chi(q-\frac{1}{4}\ve^{-1/24})\big)\\ 
		=& \, \phi_0(q)\bar\psi_\ve(q,t)\big(1-\chi(q-\frac{1}{4}\ve^{-1/24})\big)+ \big(\phi_\ve(q)-\phi_0(q)\big)\bar\psi_\ve(q,t)\big(1-\chi(q-\frac{1}{4}\ve^{-1/24})\big)\\ 
		=:& \, w_{\ve 1}(q,t)+	w_{\ve 2}(q,t).
	\end{aligned}
\end{equation*}
Since 
\begin{equation*}
	\begin{aligned}
\left|\frac{\partial w_{\ve 1}}{\partial q}\right|^2   \leq&\,  3
\left|\phi'_0(q)\bar\psi_\ve(q,t)\big(1-\chi(q-\frac{1}{4}\ve^{-1/24})\big)\right|^2\\ 
+&\, 3
\left|\phi_0(q)\frac{\partial \bar\psi_\ve(q,t)}{\partial q}\big(1-\chi(q-\frac{1}{4}\ve^{-1/24})\big)\right|^2
+3\left|\phi_0(q)\bar\psi_\ve(q,t)\chi^\prime(q-\frac{1}{4}\ve^{-1/24})\right|^2, 
\end{aligned}
\end{equation*}
we can employ Proposition \ref{Prop81}, bounds \eqref{boundq}, \eqref{boundt} (that hold with $\phi_0$ in place of $\phi_\ve$ due to the construction of $\phi_\ve$) and \eqref{boundt2} to show that
\begin{equation*}
	\begin{aligned}
		\int\limits_{0}^{\infty} 
	\int\limits_{-\frac{1}{2}\ve^{-1/24}}^{\frac{1}{2}\ve^{-1/24}}	\left|\nabla w_{\ve 1}(q,t)\right|^2  dtdq 
	+	\int\limits_{0}^{\infty} 
	\int\limits_{-\frac{1}{2}\ve^{-1/24}}^{\frac{1}{2}\ve^{-1/24}}	 w^2_{\ve 1}(q,t)(1+q+t^2)  dtdq
	\leq C.
	\end{aligned}
\end{equation*}
It follows that the sequence of the functions $w_{\ve 1}(q,t)$
(restricted to $[0,\infty)\times [-\frac{1}{2}\ve^{-1/24},\frac{1}{2}\ve^{-1/24}]$ and extended by zero elswhere)  is compact in $L^2([0,\infty)\times\mathbb{R})$.

	Next we show that the $L^2$-norm of the function 
	$(\phi_\ve(q)-\phi_0(q))\, \bar\psi_\ve(q,t)\, \big(1-\chi(q-\frac{1}{4}\ve^{-1/24})\big)$ tends to zero as $\ve\to 0$. To this end observe that
$$
|\phi_\ve(q)-\phi_0(q)|\leq C\ve^{1/3}(1-\chi(q-1)+\sqrt{q}\, \phi_0(q))\quad \text{for} \ \ 0\leq q \leq  \frac{1}{4}\ve^{-1/24}+1.
$$
and prove that the functions $z_\ve:=\ve^{1/3}\bar\psi_\ve(q,t)\,(1-\chi(q-1))$ and $\ve^{1/3}\sqrt{q}\, \phi_0(q) \bar\psi_\ve(q,t)\, \big(1-\chi(q-\frac{1}{4}\ve^{-1/24})\big)$ 
converge to zero (in $L^2$-norm). Since $\sqrt{q}\leq 1/\ve^{1/48}$ on the support of the latter functions, the sought convergence of these functions follows 
from \eqref{boundt2} and \eqref{phipsi}. Regarding  functions $z_\ve$, we use the following version of the Poincar\'e inequality
\begin{equation}
	\begin{aligned}
		\int\limits_0^2\, |z_\ve|^2dq\leq C\left(
		\int\limits_0^2\, \left|\frac{\partial z_\ve}{\partial q}\right|^2dq+ \left|\int\limits_0^2\,z_\ve\phi_0dq \right|^2 \right) 
		\label{Poincare}
	\end{aligned}
\end{equation}
Computing the partial derivative of $z_\ve$ in $q$ and using   \eqref{boundt2}, \eqref{phipsi},\eqref{boundq} we derive 
\begin{equation}
	\begin{aligned}
\int\limits_{-\frac{1}{2}\ve^{-1/24}}^{\frac{1}{2}\ve^{-1/24}}&
\int\limits_0^2
\left|\frac{\partial z_\ve}{\partial q}\right|^2dqdt\\ 
&\leq 
C\int\limits_{-\frac{1}{2}\ve^{-1/24}}^{\frac{1}{2}\ve^{-1/24}}\int\limits_0^2
\left|\frac{\partial \bar\psi_\ve}{\partial q}\right|^2\phi^2_\ve(q)dqdt\
+\ve^{2/3}
\int\limits_{-\frac{1}{2}\ve^{-1/24}}^{\frac{1}{2}\ve^{-1/24}}
\int\limits_1^2
|\bar\psi_\ve|^2|\chi^\prime(q-1)|^2dqdt\leq C \ve^{1/3}.
\label{8.24}
\end{aligned}
\end{equation}
We apply the Cauchy-Schwarz inequality to the second term in \eqref{Poincare}  and use \eqref{boundt2}, \eqref{phipsi} to get  
\begin{equation}
	\begin{aligned}
\int\limits_{-\frac{1}{2}\ve^{-1/24}}^{\frac{1}{2}\ve^{-1/24}}
	\left|\int\limits_0^2\,z_\ve\phi_0dq \right|^2dt\leq 
	2\int\limits_{-\frac{1}{2}\ve^{-1/24}}^{\frac{1}{2}\ve^{-1/24}}
	\int\limits_0^2\,z_\ve^2\phi_0^2dqdt
	\leq C\ve^{2/3}
	\int\limits_{-\frac{1}{2}\ve^{-1/24}}^{\frac{1}{2}\ve^{-1/24}}
	\int\limits_0^2 \bar\psi_\ve^2\phi^2_\ve(q)dqdt\leq C_1\ve^{2/3}.
	\label{8.25}
\end{aligned}
\end{equation}
Combining \eqref{Poincare}, \eqref{8.24} and \eqref{8.25} we obtain that $L^2$-norm of $z_\ve$ tends to zero.

Thus, Theorem \ref{Th41} is completely proved.

\subsection{$\Gamma$-$\liminf$ inequality}
\label{lowerboundineq}

In this section we establish the $\Gamma$-$\liminf$ inequality
\begin{equation*}
	\liminf_{\ve\to 0} F_\ve(\rho_\ve)\geq F_0(\rho_0)
\end{equation*}
for any sequence of functions $\rho_\ve$ that converges, with respect to the convergence defined in Theorem \ref{Th41}, to $\rho_0$, as $\ve\to 0$. Due to Theorem \ref{Th41}, the limit function $\rho_0$ necessarily has the form $\rho_0(q,t)=\phi_0(q)\psi_0(t)$, where $\psi_0\in H^1(\mathbb{R})$. Moreover, in view of Lemma \ref{Lemma72}, it is  sufficient to prove the $\Gamma$-$\liminf$ inequality in the case when $\psi_0\in C_c^\infty(\mathbb{R})$.  
We use the following bound
\begin{equation*}
	\begin{aligned}
F_\ve(\rho_\ve)&\geq \Upsilon_\ve(\psi_\ve)\\
&+\frac{1}{\ve}
\int\limits_{0}^{\frac{1}{4}\ve^{-1/24}} 
\int\limits_{-\frac{1}{2}\ve^{-1/24}}^{\frac{1}{2}\ve^{-1/24}}
\left(m\big(X(\ve^{2/3} q,\ve^{1/2}t)\big)-m(0) -\alpha\ve^{2/3} q +\ve+
\ve^{2/3} R_\ve(q,\ve^{1/2}t)\right)\\
&	\times\phi^2_\ve(q)\psi^2_\ve(q,t)
j_\ve(q,\ve^{1/2} t)dtdq,
	\end{aligned}
\end{equation*}
where 
\begin{equation}
	\begin{aligned}
		&\Upsilon_\ve(\psi_\ve) =
		\frac{1}{2\ve^{11/6}} \int\limits_{0}^{\frac{1}{4}\ve^{-1/24}} \int\limits_{0}^{\frac{1}{4}\ve^{-1/24}} 
		\int\limits_{-\frac{1}{2}\ve^{-1/24}}^{\frac{1}{2}\ve^{-1/24}}   \int\limits_{-\frac{1}{2}\ve^{-1/24}}^{\frac{1}{2}\ve^{-1/24}} 
		J\left( 
		\frac{|X(\ve^{2/3} q,\ve^{1/2}t)-X(\ve^{2/3}\tilde q,\ve^{1/2}\tilde t)|^2}{\ve^2}\right)\\ 
		&\quad\quad\times\left( \psi_\ve\left(q,t\right) 
		-\ \psi_\ve\left(\tilde q,\tilde t \right)
		\right)^2 \phi_\ve( q) j_\ve(q,\ve^{1/2} t)\phi_\ve(\tilde q)j_\ve(\tilde q,\ve^{1/2}\tilde t) dtd\tilde t dqd\tilde q. 
			\label{Upsilon}
	\end{aligned}
\end{equation}

Write the functional \eqref{Upsilon} in the form
\begin{equation}
	\begin{aligned}
		\Upsilon_\ve(\psi_\ve)&=\Upsilon_\ve(\psi_0)+\Upsilon_\ve(\psi_\ve-\psi_0)\\
		&+2
\int\limits_{0}^{\frac{1}{4}\ve^{-1/24}}
	\int\limits_{-\frac{1}{2}\ve^{-1/24}}^{\frac{1}{2}\ve^{-1/24}}  
 (A_\ve\psi_0)(q,t)\left(\psi_\ve(q,t)-\psi_0(t)\right) \phi_\ve(q)j_\ve( q,\ve^{1/2}t)dtdq,
 \label{9.3}
\end{aligned}
\end{equation}
where
\begin{equation*}
\begin{aligned}
 A_\ve\psi_0=	\frac{1}{\ve^{11/6}} &\int\limits_{0}^{\frac{1}{4}\ve^{-1/24}}
 \int\limits_{-\frac{1}{2}\ve^{-1/24}}^{\frac{1}{2}\ve^{-1/24}}  
J\left( 
\frac{|X(\ve^{2/3} q,\ve^{1/2}t)-X(\ve^{2/3}\tilde q,\ve^{1/2}\tilde t)|^2}{\ve^2}\right)\\ 
&\quad\quad\times\left( \psi_0\left(t\right) 
-\ \psi_0\left(\tilde t \right)
\right) \phi_\ve(\tilde q)j_\ve(\tilde q,\ve^{1/2}\tilde t)  d\tilde t d\tilde q.
\end{aligned}
\end{equation*} 

Arguing as in the proof of Lemma \ref{Lemma71}, one can show that $\Upsilon_\ve(\psi_0)\mathop{\longrightarrow}\limits_{\ve\to 0}\,  \frac{\theta}{2} \int\limits_{-\infty}^\infty
\big|\psi'_0(t)\big|^2dt$. The second term in 
\eqref{9.3} is non-negative.
In view of the convergence property of $\phi_\ve\psi_\ve$ and construction of the function $\phi_\ve$, the norm of $
(\psi_\ve-\psi_0)\phi_\ve$ in $L^2\left([0,\frac{1}{4}\ve^{-1/24}]\times[-\frac{1}{2}\ve^{-1/24}, \frac{1}{2}\ve^{-1/24]}\right)$
tends to zero. Hence, it suffices to show that $A_\ve\psi_0$ has uniformly bounded $L^2$-norm.
 
Perform the change of variables $\tilde t-t=\ve^{1/2}\tau$, $\tilde q-q=\ve^{1/3}\xi$ and expand 
$\psi_0(\tilde t)=\psi_0(t)+\psi'_0(t)\ve^{1/2}\tau+ \frac{1}{2}\psi''_0(t^*)\ve\tau^2$. Using \eqref{yadro} 
 we get
\begin{equation*}
	\begin{aligned}
(A_\ve\psi_0)(q,t)&= \frac{1}{\ve}  \int\limits_{-q\ve^{-1/3}}^{\frac{1}{4}\ve^{-1/24}-q\ve^{-1/3}}\int\limits_{-\infty}^\infty 
\left(J\left( 
\tau^2j_\ve^2(q+\ve^{1/3}\xi,t+\ve^{1/2}\tau)+\xi^2\right)+\ve \mathcal O_\ve(q,t,\xi,\tau)\right)
\\ 
&\times \left(-\psi'_0(t)\ve^{1/2}\tau-\frac{1}{2}\psi''_0(t^*)\ve\tau^2\right) \phi_\ve(q+\ve^{1/3}\xi)j_\ve(q+\ve^{1/3}\xi,\ve^{1/2}t+\ve \tau) 
d\tau d\xi,
\end{aligned}
\end{equation*} 
where $ \mathcal O_\ve(q,t,\xi,\tau)$ is a uniformly bounded function vanishing when $|\xi|>c$ or $|\tau|>c$ for some $c>0$ independent of $\ve$.   
Then, in view of \eqref{jacobian} we obtain
\begin{equation*}
	\begin{aligned}
		\big|(A_\ve\psi_0)(q,t)\big|\leq C \left(\big|\psi'_0(t)\big|+\big|\psi''_0(t^{**})\big|\right) \phi_\ve(q)
	\end{aligned}
\end{equation*} 
and
\begin{equation*}
	\int\limits_{0}^{\frac{1}{4}\ve^{-1/24}}
	\int\limits_{-\frac{1}{2}\ve^{-1/24}}^{\frac{1}{2}\ve^{-1/24}}
		\big|(A_\ve\psi_0)(q,t)\big|^2 dtdq
		\leq C.
\end{equation*}

Thus, taking $\mathop{\liminf}\limits_{\ve \to 0}$ in \eqref{9.3} we get
\begin{equation*}
			\liminf_{\ve \to 0} \Upsilon_\ve(\psi_\ve)
		\geq \frac{\theta}{2} \int\limits_{-\infty}^\infty
		\big|\psi'_0(t)\big|^2dt.
\end{equation*}
Theorem \ref{Th43} is completely proved. 


\begin{thebibliography}{30}

\bibitem{AllPia2002} 
Allaire, G., Piatnitski, A.:  
Uniform spectral asymptotics for singularly perturbed locally periodic operators. 
Commun. Partial Differ. Equ., 27(3-4), 705–725 (2002)

\bibitem{BatZha2007} 
Bates, P. W., Zhao, G.: Existence, uniqueness and stability of the stationary solution to a nonlocal evolution equation arising in population dispersal.
J. Math. Anal. Appl. , 332(1), 428–440 (2007)


\bibitem{BerCovHoa2016} Berestycki, H., Coville, J., Vo, H.-H.: On the definition and the properties of the principal eigenvalue of some nonlocal operators.
J. Funct. Anal. 271(10), 2701–2751. (2016)  DOI:10.1016/j.jfa.2016.05.017 

\bibitem{BraPia2021} Braides, A., Piatnitski, A.: Homogenization of random convolution energies.  J. London Math. Soc. 104, no.1, 295-319 (2021)

\bibitem{Chen_etal2024} 
Chen, L., Painter, K., Surulescu, C., Zhigun, A.: Mathematical models for cell migration: a non-local perspective. 
Phil. Trans. R. Soc. B 375(1807), 20190379 (2020)


\bibitem{CroKim1970} Crow, J. F., Kimura, M.:  An Introduction to Population Genetics Theory. Harper \& Row, 1970


\bibitem{CorEtgRosWol2008} Cortazar, C., Elgueta, M., Rossi,J.D.,  Wolanski, N.: How to approximate the heat equation with
Neumann boundary conditions by nonlocal diﬀusion problems. Arch. Ration. Mech. Anal. 187 (1),  137–156 (2008)

\bibitem{HutMarMisVic2003}
Hutson, V., Martinez, S., Mischaikow, K., Vickers, G. T.: The evolution of dispersal. J. Math. Biol. 47(6), 483–517 (2003)


\bibitem{OthDunAlt1988} Othmer, H. G., Dunbar, S. R., Alt, W.: Models of dispersal in biological systems. J. Math. Biol. 26, 263–298 (1988)


\bibitem{LiCovWan2017}  Li, F., Coville, J., Wang, X.: 
On eigenvalue problems arising from nonlocal diffusion models. 
Discrete Contin. Dyn. Sys. 37(2), 879-903 (2017) 
doi: 10.3934/dcds.2017036


\bibitem{Cov2010}  Coville, J.: 
On a simple criterion for the existence of a principal eigenfunction of some nonlocal operators. 
J. Differ. Equ., 249(11), 2921–2953 (2010)

\bibitem{PiaZhi2019} Piatnitski, A., Zhizhina, E.: Homogenization of biased convolution type operators. 
Asymptot. Anal. 115 (3-4), 241–262 (2019)

\bibitem{PiaPirZhi2022} Piatnitski, A., Pirogov, S., Zhizhina, E.: 
Large deviations for Markov jump processes in periodic and locally periodic environments. Ann. Appl. Probab. 32 (6), 
 4611–4641 (2022)


\bibitem{PiaRyb2016} Piatnitski, A., Rybalko, V.: 
On the first eigenpair of singularly perturbed operators with oscillating coefficients. Commun. Partial Differ. Equ.  41 (1), 
1–31 (2016)

\bibitem{PiaZhi2017} 
Piatnitski, A., Zhizhina, E.: Periodic homogenization of nonlocal operators with a convolution-type kernel.
SIAM  J. Math. Anal. 49 (1), 64–81 (2017)

\bibitem{Cov2013}
Coville, J.: Singular measure as principal eigenfunction of some nonlocal operators. Appl. Math. Lett. 26(8), 831-835 (2013)

\bibitem{PiaRyb2024} Piatnitski, A., Rybalko, V.:  Homogenization of nonlocal spectral problems. Z. Angew. Math. Phys. 75(6), 226 (2024)

\end{thebibliography}
\end{document}